\documentclass[11pt]{article}
\usepackage[english]{babel}
\usepackage{amsmath,amsthm,amssymb}
\usepackage{mathrsfs}
\usepackage{bbm}
\usepackage{amsfonts}
\usepackage[margin=0.8in]{geometry}
\usepackage{graphicx}
\usepackage{caption}
\usepackage{subcaption}
\usepackage{wasysym}
\usepackage{enumerate}
\usepackage{algorithm2e}
\usepackage{tabularx}
\usepackage{color}
\usepackage{wrapfig}
\usepackage{todonotes}

\newtheorem{thm}{Theorem}[section]
\newtheorem{ass}[thm]{Assumption}

\newtheorem{lem}[thm]{Lemma}
\newtheorem{prop}[thm]{Proposition}

\theoremstyle{definition}
\newtheorem{defn}[thm]{Definition}

\theoremstyle{rem}
\newtheorem{rem}[thm]{Remark}

\numberwithin{equation}{section}
\newcommand{\R}{\mathbb R}
\newcommand{\eps}{\varepsilon}

\newcommand{\bbC}{\mathbb C}
\newcommand{\bfC}{\mathbf C}
\newcommand{\bfU}{\mathbf U}

\newcommand{\bbF}{\mathbb F}
\newcommand{\bbT}{\mathbb T}

\newcommand{\bbN}{\mathbb N}

\newcommand{\mcA}{\mathcal{A}}

\newcommand{\mcB}{\mathcal{B}}
\newcommand{\mcC}{\mathcal C}

\newcommand{\mcE}{\mathcal E}

\newcommand{\mcN}{\mathcal N}
\newcommand{\mcF}{\mathcal F}

\newcommand{\mcT}{\mathcal T}
\newcommand{\mcP}{\mathcal P}

\newcommand{\mcH}{\mathcal H}

\newcommand{\mcU}{\mathcal U}

\newcommand{\mcR}{\mathcal R}
\newcommand{\mcS}{\mathcal S}

\newcommand{\mcV}{\mathcal V}

\newcommand{\mcO}{\mathcal{O}}
\newcommand{\E}{\mathbb{E}}
\newcommand{\Prob}{\mathbb{P}}

\newcommand{\bx}{\bold{x}}
\newcommand{\by}{\bold{y}}
\newcommand{\esssup}{\mathop{\rm{ess}\sup}}
\newcommand{\essinf}{\mathop{\rm{ess}\inf}}

\newcommand{\argmin}{\mathop{\arg\min}}
\newcommand{\ett}{\mathbbm{1}}

\newcommand{\cadlag}{c\`adl\`ag~}

\newcommand{\Pred}{\mathcal{P}}
\newcommand{\Prog}{{\rm Prog}}

\newcommand{\ie}{\textit{i.e.\ }}
\newcommand{\eg}{\textit{e.g.\ }}

\begin{document}

\title{A Nonlinear Snell Envelope Representation for Path-Dependent Controller-Stopper Games}

\author{Magnus Perninge\footnote{M.\ Perninge is with the Department of Mathematics, M\"alardalen University, V\"aster{\aa}s,
Sweden. e-mail: magnus.perninge@mdu.se.}} %
\maketitle
% ----------------------------------------------------------------
\begin{abstract}
We consider a finite-horizon, zero-sum stochastic differential game in which one player controls a path-dependent stochastic system, while the opponent is given the opportunity to terminate the game prematurely. We introduce a control randomization formulation, which allows us to establish that the upper and lower value functions coincide.

Our approach also yields a representation of the common game value in terms of a nonlinear Snell envelope, where the underlying stopped process is given by the unique maximal solution of a backward stochastic differential equation (BSDE) with constrained jumps.
\end{abstract}

% ----------------------------------------------------------------
\section{Introduction}
We consider a game between two players, a controller and a stopper, where the controller influences the dynamics of a path-dependent stochastic system,
\begin{align}\label{ekv:fwd-sde}
    X^{t,\bx,u}_s=\bx(s\wedge t)+\int_t^{s\vee t} a(r,X^{t,\bx,u},u_r)ds+\int_t^{s\vee t}\sigma(r,X^{t,\bx,u},u_r)dW_r,\quad\forall s\in [0,T],
\end{align}
having history $\bx\in C([0,t]\to\R^d)$ at time $t$, by choosing a progressively measurable control process $(u_s:t\leq s\leq T)$, taking values in the compact set $U\subset\R^d$. The controller's objective is to minimize her cost functional
\begin{align*}
  J(t,\bx;u,\tau^S(u))&:=\E\Big[\psi(\tau^S(u),X^{t,\bx,u})+\int_t^{\tau^S(u)}f(s,X^{t,\bx,u},u_s)ds\,\Big|\,\mcF_t\Big].
\end{align*}
The stopper, on the other hand, seeks to maximize the same quantity by selecting a stopping strategy $\tau^S \in \mathcal{T}^S_t$, the latter being the set of non-anticipative maps from the set of admissible controls, $\mcU_t$, into the set of stopping times $\tau\geq t$, denoted by $\mcT_t$.

The lower value of this game is defined as
\begin{align}\label{ekv:lower-vf-def}
  \underline v(t,\bx)=\essinf_{u\in\mcU_t}\esssup_{\tau\in\mcT_t}J(t,\bx;u,\tau)
\end{align}
whereas the upper value is given by
\begin{align}\label{ekv:upper-vf-def}
  \bar v(t,\bx)=\esssup_{\tau^S\in\mcT^S_t}\essinf_{u\in\mcU_t}J(t,\bx;u,\tau^S(u)).
\end{align}
Well-posedness of the game amounts to existence of a value, that is,
\begin{align}\label{ekv:zsg-intro}
  \underline v(t,\bx)=\bar v(t,\bx).
\end{align}
We prove that \eqref{ekv:zsg-intro} holds by showing that both the upper and lower value functions coincide with the value of a dual controller-stopper game, where we replace the control $u$ by a Poisson point process
\begin{align*}
  I^t_s&=\beta_0+\int_t^s\!\!\int_U(e-I_{r-})\mu(dr,de),\quad\forall s\in [t,T].
\end{align*}
In the dual game, rather than directly choosing a control, the controller manipulates the distribution of $I^t$ by selecting the compensator of the random measure $\mu$. The stopper, on the other hand, selects a stopping time from $\mcT^\mcR_t$, the set of stopping times $\tau\geq t$ with respect to the filtration generated by $W$ and $\mu$. Since this filtration contains the information generated by $I^t$, the dual game naturally mirrors the information structure induced by non-anticipative stopping strategies in the original game.

Our main result, together with the alternative representation of the dual game established in \cite{imp-stop-game}, yields the nonlinear Snell envelope representation
\begin{align}\label{ekv:nonlin-Snell}
  \underline v(t,\bx)=\bar v(t,\bx)=\esssup_{\tau\in\mcT^\mcR_t} Y^{t,\bx;\tau}_t,
\end{align}
where for each $\tau\in\mcT^\mcR_t$, the process $Y^{t,\bx;\tau}$ is the first component of the unique maximal solution to a BSDE with constrained jumps.

\medskip

\textbf{Related literature} The controller-stopper game, first introduced in \cite{MaitraSudderth}, has been extensively studied in the literature, with a clear emphasis on the zero-sum version of the game \cite{Karatzas2001,Weerasinghe2006,Zamfirescu08,BayraktarHuang2013,BayraktarYao14,NutzZhang15,Choukroun15,cont-stop-P2}. Among these, \cite{Zamfirescu08}, \cite{BayraktarYao14}, and \cite{NutzZhang15} study path-dependent, non-degenerate models. The former by applying a martingale approach and restricting attention to cases with uncontrolled volatility, whereas the latter two exploit methods of quasi-sure analysis and adopt a weak formulation in which the controller selects a probability measure on path space, allowing for models with controlled volatility. References \cite{Karatzas2001} and \cite{Weerasinghe2006} restrict attention to one-dimensional diffusions, the latter allowing for degeneracy in the volatility, while \cite{BayraktarYao14} and \cite{Choukroun15} both consider a general multi-dimensional Markovian framework. Notably, \cite{Choukroun15} establishes a link between reflected BSDEs with constrained jumps and nonlinear variational inequalities. Exploiting the well-known connection between nonlinear variational inequalities and controller–stopper games, the authors derive a dual characterization of Markovian controller–stopper games in terms of reflected BSDEs with constrained jumps. In \cite{cont-stop-P2}, which is a companion paper to the present work, we build on the nonlinear Snell envelope representation \eqref{ekv:nonlin-Snell} to establish that the controller-stopper game has a value even when the stopper is restricted to feedback stopping rules.

%Related to controller–stopper games are control problems that combine control and stopping, which have received considerable attention over the past decades. In \cite{Zamfirescu06,Bayraktar2011,ElAsri2020,Benes2025}, such problems are studied under various assumptions on the model coefficients.
%
%Controller–stopper games themselves have also been generalized in several directions. An interesting strand of the literature extends the classical framework by considering situations in which the presence of a competing player is unknown to either the controller or the stopper~\cite{EkstromSalami,Bodnariu2024,EkstromDeFinetti}.

As described above, we seek BSDE representations for the value of controller-stopper games. The ensemble of approaches that use BSDEs with constrained jumps to represent the value process in stochastic optimal control problems and stochastic differential games is commonly referred to as \emph{control randomization}. The control randomization methodology was initiated by the seminal work \cite{Kharroubi2010}, which related the solution of a BSDE with constrained jumps to the unique viscosity solution of a quasi-variational inequality (see also the related work on stochastic target problems in \cite{Bouchard09}). Particularly relevant for the present paper is \cite{KharroubiPham15}, which represents viscosity solutions of Hamilton–Jacobi–Bellman integro-partial differential equations (HJB-IPDEs) by means of BSDEs with constrained jumps and serves as a foundation for \cite{Choukroun15}.

The above-mentioned works only achieve a representation of the value in control problems by deriving a Feynman-Kac representation of the corresponding PDEs. A significant breakthrough in control randomization was achieved in \cite{Fuhrman15}, which directly links the value function of the randomized control problem to that of the original control problem. This eliminated the need for a Feynman-Kac representation, thereby extending the framework to encompass stochastic systems with path-dependencies. Building on this foundation, subsequent advancements adapted the methodology to partial information settings \cite{Bandini18} and optimal switching problems \cite{Fuhrman2020}.%, while a numerical approximation algorithm was developed in \cite{Kharroubi14}.% More recently, \cite{imp-stop-game} further expanded the scope of control randomization by establishing a link between optimal stopping for BSDEs with constrained jumps and zero-sum games of impulse control versus stopping.

\medskip

\textbf{Contribution} We establish existence of a game value in a general path-dependent framework under mild growth and regularity conditions on the coefficients. In contrast to \cite{Zamfirescu08,BayraktarYao14,NutzZhang15}, our dual formulation accommodates degenerate diffusions as well as reward coefficients with polynomial growth in the supremum norm of the state process. Moreover, even in the Markovian setting, our assumptions are weaker, as we do not require continuity of the terminal reward function $\psi$.

Finally, our analysis extends the control randomization literature by establishing a direct link between controller-stopper games and optimal stopping problems for BSDEs with constrained jumps.

\medskip

\textbf{Outline} The next section provides some preliminary definitions and sets the main assumptions that hold throughout the paper. In Section~\ref{sec:zs}, we state the main results (summarized in Theorem~\ref{thm:zs-game}) and discuss its components, primarily by describing the dual version of the game. Section~\ref{sec:proof} then gives a detailed proof of Theorem~\ref{thm:zs-game}.

\section{Preliminaries}
\subsection{Probabilistic setup}
We let $(\Omega,\mcF,\Prob)$ be a complete probability space supporting a $d$-dimensional Brownian motion, denoted by $W$, and an independent Poisson random measure $\mu$  on $[0,T]\times U$ with compensator $dt\otimes \lambda(de)$, where $\lambda$ is a finite measure on $U$ with full topological support. We denote by $\bbF:=(\mcF_t)_{t\geq 0}$ the augmented natural filtration generated by $W$, while $\bbF^{\mcR}:=(\mcF^{\mcR}_t)_{t\geq 0}$ is the augmented natural filtration generated by both $W$ and $\mu$.

In this setting, for each $E\in\mathcal{B}(U)$, the compensated process
\[
\tilde{\mu}([0,t],E) := \mu([0,t],E) - t\lambda(E), \quad t\geq 0,
\]
is an $\mathbb{F}^{\mathcal R}$-martingale.

\begin{rem}
One could alternatively work with two separate probability spaces, say $(\Omega,\mcF,\bbF,\Prob,W)$ and $(\Omega^{\mcR},\mcF^{\mcR},\bbF^{\mcR},\Prob^{\mcR},W^{\mcR},\mu^{\mcR})$, and formulate the primal controller-stopper game on the former and the randomized (dual) game on the latter.

While this approach would allow for slightly greater generality, we adopt the single-space formulation in order to avoid unnecessary notational complexity.
\end{rem}

\subsection{Notations}

\noindent Throughout, we will use the following notations, where $T> 0$ is the fixed problem horizon:
\begin{itemize}
  \item We denote by $\bfC^d$ the set of continuous functions $\bx:[0,T]\to\R^d$ equipped with the supremum norm $\|\cdot\|_T$, where $\|\bx\|_t:=\sup_{s\in[0,t]}|\bx(s)|$ and for $(t,\bx)\in[0,T]\times\bfC^d$, we set $\bfC^d_{t,\bx}:=\{\bx'\in\bfC^d:\bx'(s)=\bx(s),\forall s\in [0,t]\}$.
  \item We let $\bbC:=(\mcC_t)_{t\in [0,T]}$ be the filtration generated by the coordinate map, $\mcC_t:=\sigma(\bx\mapsto\bx(s):s\in [0,t])$ on $\bfC^d$. Similarly, $\mcC^{t,\bx}_s:=\sigma(\bx'\mapsto\bx'(r):r\in [t,s])$ denotes the filtration generated by the coordinate map on $\bfC^d_{t,\bx}$.
  \item For any two maps, $\bx,\by:[0,T]\to\R^d$, we define concatenation at $t\in [0,T]$ as $(\bx \otimes_t \by)(s):=\ett_{[0,t]}(s)\bx(s) + \ett_{(t,T]}(s)\by(s)$ for all $s\in [0,T]$.
  \item We let $\Lambda:=[0,T]\times \bfC^d$ and for $(t,\bx)\in\Lambda$, we set $\Lambda^{t,\bx}:=[t,T]\times \bfC^d_{t,\bx}$. We equip the set $\Lambda$ with the pseudo-metric
      \begin{align}
        \mathbf d_\Lambda[(t,\bx),(t',\bx')]:=|t'-t|+\|\bx'(\cdot\wedge t')-\bx(\cdot\wedge t)\|_T.
      \end{align}
  \item For a measure space $(\tilde\Omega,\tilde\mcF)$ and a filtration $\tilde\bbF$ on $\tilde\mcF$ we let $\Prog(\tilde\bbF)$ (resp. $\Pred(\tilde\bbF)$) denote the $\sigma$-algebra of $\tilde\bbF$-progressively (resp. $\tilde\bbF$-predictably) measurable subsets of $\R_+\times \tilde\Omega$.
  \item We let $\mcU_t$ be the set of $\Prog(\bbF)$-measurable processes $(u_s:t\leq s\leq T)$ valued in the compact set $U\subset\R^d$.
  \item We let $\mcT^F$ be the set of maps $\tau^F:\bfC^d\to [0,T]$ such that $\{\bx\in\bfC^d:\tau^F(\bx)\le t\}\in\mcC_t$ for all $t\in[0,T]$. For each $t\in[0,T]$, we let $\mcT_t^F$ denote the subset of $\mcT^F$ consisting of maps $\tau^F$ satisfying $\tau^F(\bx)\ge t$ for all $\bx\in\bfC^d$.
  \item We let $\mcT$ be the set of all $[0,T]$-valued $\bbF$-stopping times and for each $\eta\in\mcT$, we let $\mcT_\eta$ be the corresponding subset of stopping times $\tau$ such that $\tau\geq \eta$, $\Prob$-a.s.
  \item Similarly, we let $\mcT^\mcR$ be the set of all $[0,T]$-valued $\bbF^{\mcR}$-stopping times and for each $\eta\in\mcT^\mcR$, we let $\mcT^\mcR_\eta$ be the corresponding subset of stopping times $\tau$ such that $\tau\geq \eta$, $\Prob$-a.s.
  \item For $p\geq 1$, $t\in [0,T]$ and $\tau\in\mcT^\mcR_t$, we let $\mcS^{p}_{t,\tau}$ be the set of all $\R$-valued, $\Prog(\bbF^{\mcR})$-measurable \cadlag processes $(Z_s: s\in[t,\tau])$ such that $\|Z\|_{\mcS^{p,\tau}}:=\E\Big[\sup_{s\in [t,\tau]} |Z_s|^p\Big]^{1/p}<\infty$.
      When $\tau=T$, we use the shorter notation $\mcS^{p}_t$.
  \item We let $\mcA^{p}_{t,\tau}$ denote the subset of $\mcS^{p}_{t,\tau}$ consisting of all $\mcP(\bbF^{\mcR})$-measurable, nondecreasing processes $Z$ satisfying $Z_t=0$. Moreover, we let $\mcA^p_t:=\mcA^{p}_{t,T}$.
  \item We let $\mcH^{p}_{t,\tau}(W)$ denote the set of all $\R^d$-valued, $\mcP(\bbF^{\mcR})$-measurable processes $(Z_s: s\in[t,\tau])$ such that $\|Z\|_{\mcH^{p}_{t,\tau}(W)}:=\E\left[(\int_t^\tau |Z_s|^2 ds)^{p/2}\right]^{1/p}<\infty$. When $\tau=T$, we use the notation $\mcH^{p}_t(W)$.
  \item We let $\mcH^{p}_{t,\tau}(\mu)$ denote the set of all $\R$-valued, $\mcP(\bbF^{\mcR})\otimes \mcB(U)$-measurable mappings $(Z_s(e): s\in[t,\tau],e\in U)$ such that $\|Z\|_{\mcH^{p}_{t,\tau}(\mu)}:=\E\Big[\int_t^\tau\!\!\int_U |Z_s(e)|^p\lambda(de)ds\Big]^{1/p}<\infty$ and set $\mcH^{p}_t(\mu):=\mcH^{p}_{t,T}(\mu)$.
\end{itemize}

Unless otherwise stated, all inequalities involving random variables are assumed to hold $\Prob$-a.s.

The definition of the upper value function in \eqref{ekv:upper-vf-def} uses non-anticipative stopping strategies defined as follows (see also the seminal work of Elliott \& Kalton~\cite{ElliotKalton72}):
\begin{defn}
For $t\in [0,T]$, the set of non-anticipative stopping strategies, denoted $\mcT^{S}_t$, is defined as the set of maps $\tau^S:\mcU_t\to\mcT_t$ such that for any $s\in [t,T]$ and $u,\tilde u\in \mcU_t$, the sets $\{\omega:\tau^S(u)\leq s\}$ and $\{\omega:\tau^S(\tilde u)\leq s\}$ differ by at most a $\Prob$-null set, whenever $u(r)=\tilde u(r)$, $dr\otimes d\Prob$-a.e.~on $[t,s]\times\Omega$.
\end{defn}

\subsection{Assumptions}
We assume that the coefficients of the forward SDE satisfy the following conditions:
\begin{ass}\label{ass:onSDE}
\begin{enumerate}[i)]
  \item\label{ass:onSDE-a-sigma} The coefficients $a:[0,T]\times\bfC^d\times U\to\R^{d}$ and $\sigma:[0,T]\times\bfC^d\times U\to\R^{d\times d}$ have components that are $\Prog(\bbC)\otimes\mcB(U)$-measurable, continuous in $u$, uniformly on sets of the form $\{(t,\bx): \|\bx\|_t \le K\}$, for each $K>0$, satisfy the linear growth condition
  \begin{align}\label{ekv:a-sigma-growth}
    |a(t,\bx,u)|+|\sigma(t,\bx,u)|&\leq C(1+\|\bx\|_t)
  \end{align}
  and the Lipschitz continuity
  \begin{align*}
    |a(t,\bx,u)-a(t,\bx',u)|+|\sigma(t,\bx,u)-\sigma(t,\bx',u)|&\leq k_{a,\sigma}\|\bx'-\bx\|_t
  \end{align*}
  for all $(t,\bx,\bx')\in [0,T]\times \bfC^d\times\bfC^d$ and $u\in U$.
\end{enumerate}
\end{ass}

Moreover, we make the following assumptions on the coefficients in the reward functionals $J^C$ and $J^S$:

\begin{ass}\label{ass:oncoeff}
There are constants $C>0$ and $q>0$ in addition to a family of moduli of continuity $(\varpi_K)_{K\geq 0}$ such that for $i=C,S$, we have:
\begin{enumerate}[i)]
  \item\label{ass:oncoeff-f} The running cost/reward $f^i:[0,T]\times \bfC^d\times U\to\R$ is $\Prog(\bbC)\otimes\mcB(U)$-measurable and satisfies the growth condition
  \begin{align*}
    |f^i(t,\bx,u)|\leq C(1+\|\bx\|_t^q).
  \end{align*}
  Moreover, for each $K>0$,
  \begin{align*}
    |f^i(t,\bx',u')-f^i(t,\bx,u)|\leq \varpi_K(\|\bx'-\bx\|_t+|u'-u|)
  \end{align*}
  for all $(t,\bx,\bx',u,u')\in [0,T]\times \bfC^d\times\bfC^d\times U\times U$, with $\|\bx\|_t\vee \|\bx'\|_t\leq K$.
  \item The terminal reward $\psi^i:[0,T]\times\bfC^d\to\R$ is $\Prog(\bbC)$-measurable with \cadlag paths $t\mapsto \psi^i(t,\bx)$ for all $\bx\in\bfC^d$, and satisfies a polynomial growth condition, \ie
  \begin{align*}
    |\psi^i(t,\bx)|\leq C(1+\|\bx\|_t^q).
  \end{align*}
  Moreover, for every $K>0$,
  \begin{align*}
    \psi^i(t',\bx')-\psi^i(t,\bx)\leq \varpi_K(\mathbf d_\Lambda[(t,\bx),(t',\bx')]),
  \end{align*}
  whenever $0\leq t\leq t'\leq T$ and $\bx,\bx'\in\bfC^d$ satisfy $\|\bx\|_t\vee \|\bx'\|_{t'}\leq K$.
\end{enumerate}
\end{ass}

\subsection{Preliminary estimates}
The following proposition gathers preliminary estimates where $\rho_t:\mcU_t\times\mcU_t\to\R_+$ is the pseudo-metric on $\mcU_t$ defined as
\begin{align*}
  \rho_t(u,\tilde u):=\E\Big[\int_{t}^T|u_s-\tilde u_s| ds\Big]
\end{align*}
and for each $(t,\bx)\in [0,T]\times\bfC^d$ and $u\in\mcU_t$, the controlled state process $X^{t,\bx;u}$ solves the forward SDE
\begin{align*}
  X^{t,\bx;u}_s=\bx(s\wedge t)+\int_t^{s\vee t} a(r,X^{t,\bx;u},u_r)dr+\int_t^{s\vee t}\sigma(r,X^{t,\bx;u},u_r)dW_r,\quad\forall s\in [0,T].
\end{align*}

\begin{prop}\label{prop:Xu-stab}
  For any $p\geq 1$, there is a $C_p>0$ such that
  \begin{align}\label{ekv:Xu-growth}
    \E\Big[\sup_{s\in[t,T]}|X^{t,\bx;u}_s|^p\,\Big|\,\mcF_t\Big]\leq C_p(1+\|\bx\|^p_t),\quad\Prob-\text{a.s.}
  \end{align}
  for all $u\in\mcU_t$. Moreover, there is a $C$ such that for any $(t,\bx),(\tilde t,\tilde \bx)\in\Lambda$ and any sequences $(u^i)_{i\in\bbN}\subset\mcU_t$ and $(\tilde u^i)_{i\in\bbN}\subset\mcU_{\tilde t}$ such that $\rho_{t\vee\tilde t}(u^i,\tilde u^i)\to 0$, we have
  \begin{align}\label{ekv:Xu-stab}
  \limsup_{i\to\infty}\E\Big[\sup_{s\in [0,T]}|X^{t,\bx;u^i}_s- X^{\tilde t,\tilde\bx;\tilde u^i}_s|^2\Big]\leq C(|t-\tilde t|(1+\|\bx\|^2_t+\|\tilde\bx\|^2_{\tilde t})+\|\bx(\cdot\wedge t)-\tilde\bx(\cdot\wedge \tilde t)\|^2_T).
\end{align}
\end{prop}

\begin{proof}
Under \eqref{ekv:a-sigma-growth}, the moment bound follows by standard arguments (see \eg \cite{Protter}).

We prove the stability property. Let $X^i := X^{t,\bx;u^i}$ and $\tilde X^i := X^{\tilde t,\tilde\bx;\tilde u^i}$, and set $\Delta^i_s := X^i_s-\tilde X^i_s$. We assume without loss of generality that $t\leq \tilde t$ and have that
\begin{align*}
  \Delta^i_s= \bx(t)+\int_t^s a(r,X^i,u^i_r)dr+\int_t^s\sigma(r,X^i,u^i_r)dW_r-\tilde \bx(s),\quad \forall s\in [t,\tilde t].
\end{align*}
Using \eqref{ekv:a-sigma-growth} and the Burkholder--Davis--Gundy (BDG) inequality, we find that\footnote{Throughout, $C$ will denote a generic positive constant that may change value from line to line.}
\begin{align*}
  \limsup_{i\to\infty}\E\Big[\sup_{s\in [0,\tilde t]}|\Delta^i_s|^2\Big]
  &\leq C(|t-\tilde t|(1+\|\bx\|^2_t)+\|\bx(\cdot\wedge t)-\tilde\bx(\cdot\wedge \tilde t)\|^2_T).
\end{align*}
Now,
\begin{align*}
  \Delta^i_s
  &= \Delta^i_{\tilde t}+\int_{\tilde t}^s \big(a(r,X^i,u^i_r)-a(r,\tilde X^i,\tilde u^i_r)\big)dr  + \int_{\tilde t}^s \big(\sigma(r,X^i,u^i_r)-\sigma(r,\tilde X^i,\tilde u^i_r)\big)dW_r,\quad\forall s\in[\tilde t,T].
\end{align*}
Using Lipschitz continuity in the state variable,
\begin{align*}
  |a(r,X^i,u^i_r)-a(r,\tilde X^i,\tilde u^i_r)| &\leq |a(r,X^i,u^i_r)-a(r,\tilde X^i,u^i_r)|+|a(r,\tilde X^i,u^i_r)-a(r,\tilde X^i,\tilde u^i_r)|
  \\
  &\leq k_{a,\sigma} \|\Delta^i\|_r + |a(r,\tilde X^i,u^i_r)-a(r,\tilde X^i,\tilde u^i_r)|
\end{align*}
and similarly for $\sigma$. Applying the BDG inequality gives that for any $t'\geq \tilde t$,
\begin{align*}
  \E\Big[\sup_{s\in [\tilde t,t']}|\Delta^i_s|^2\Big]
  &\leq C\E\Big[|\Delta^i_{\tilde t}|^2+\int_{\tilde t}^{t'} (k_{a,\sigma} \|\Delta^i\|^2_r + |a(r,\tilde X^i,u^i_r)-a(r,\tilde X^i,\tilde u^i_r)|^2 \\
  &\quad + |\sigma(r,\tilde X^i,u^i_r)-\sigma(r,\tilde X^i,\tilde u^i_r)|^2)dr\Big].
\end{align*}
Gr\"onwall's lemma then gives that
\begin{align*}
  \E\Big[\sup_{s\in [\tilde t,T]}|\Delta^i_s|^2\Big]
  &\leq C\E\Big[|\Delta^i_{\tilde t}|^2+\int_{\tilde t}^{T} (|a(r,\tilde X^i,u^i_r)-a(r,\tilde X^i,\tilde u^i_r)|^2  + |\sigma(r,\tilde X^i,u^i_r)-\sigma(r,\tilde X^i,\tilde u^i_r)|^2)dr\Big].
\end{align*}
We proceed by noting that by \eqref{ekv:Xu-growth} and Markov's inequality, there is for each $\eps>0$ a $K>0$ such that
\begin{align*}
  &\E\Big[\int_{\tilde t}^{T} (|a(r,\tilde X^i,u^i_r)-a(r,\tilde X^i,\tilde u^i_r)|^2  + |\sigma(r,\tilde X^i,u^i_r)-\sigma(r,\tilde X^i,\tilde u^i_r)|^2)dr\Big]
  \\
  &\leq\E\Big[\ett_{[\|\tilde X^i\|_{T}\leq K]}\Big(\int_{\tilde t}^{T} (|a(r,\tilde X^i,u^i_r)-a(r,\tilde X^i,\tilde u^i_r)|^2  + |\sigma(r,\tilde X^i,u^i_r)-\sigma(r,\tilde X^i,\tilde u^i_r)|^2)dr\Big)\Big]+\eps,
\end{align*}
for all $i\in\bbN$. On the other hand, $a$ and $\sigma$ are continuous in $u$, uniformly in $(t',\bx')\in \{(s,\by)\in\Lambda:\|\by\|_{s}\leq K\}$, which implies that
\begin{align*}
  \limsup_{i\to\infty}\E\Big[\int_{\tilde t}^T |a(r,\tilde X^i,u^i_r)-a(r,\tilde X^i,\tilde u^i_r)|^2 dr + \int_{\tilde t}^T
  |\sigma(r,\tilde X^i,u^i_r)-\sigma(r,\tilde X^i,\tilde u^i_r)|^2 dr \Big]
  \leq \eps.
\end{align*}
Finally, since $\eps>0$ was arbitrary, we conclude that
\begin{align*}
  \limsup_{i\to\infty}\E\Big[\sup_{s\in [\tilde t,T]}|\Delta^i_s|^2\Big]
  &\leq C\limsup_{i\to\infty}\E\Big[|\Delta^i_{\tilde t}|^2\Big]
\end{align*}
and the result follows.
\end{proof}

\section{Dual representation of zero-sum controller-stopper games\label{sec:zs}}

The main result of the present work is summarized in the following theorem.
\begin{thm}\label{thm:zs-game}
  There exists a deterministic, continuous function $v:[0,T]\times \bfC^d\to\R$ such that
  \begin{align*}
     v(t,\bx)=\underline v(t,\bx)=\bar v(t,\bx),\quad \forall (t,\bx)\in [0,T]\times \bfC^d.
  \end{align*}
  In particular, the zero-sum game has a value. Moreover, the game value admits the following dual representation:
  \begin{align}\label{ekv:zs-nonlin-Snell}
    v(t,\bx)=\esssup_{\tau\in\mcT^\mcR_t}Y^{t,\bx;\tau}_t,\quad\forall (t,\bx)\in [0,T]\times \bfC^d,
  \end{align}
  where for each $\tau\in\mcT^\mcR_t$, the process $Y^{t,\bx;\tau}$ is the first component in the quadruple of processes\\ $(Y^{t,\bx;\tau},Z^{t,\bx;\tau},V^{t,\bx;\tau},K^{-,t,\bx;\tau})\in\mcS^2_{t,\tau}\times\mcH^2_{t,\tau}(W)\times \mcH^2_{t,\tau}(\mu)\times \mcA^2_{t,\tau}$ which constitutes the unique maximal solution to the BSDE with constrained jumps
  \begin{align}\label{ekv:bsde-c-jmp}
    \begin{cases}
      Y^{t,\bx;\tau}_s=\psi(\tau,X^{t,\bx})+\int_s^\tau f(r,X^{t,\bx},I^t_r)dr-\int_s^\tau Z^{t,\bx;\tau}_r dW_r-\int_s^\tau\!\!\int_U V^{t,\bx;\tau}_r(e)\mu(dr,de)
      \\
      \quad-( K^{-,t,\bx;\tau}_\tau-K^{-,t,\bx;\tau}_s),\quad\forall s\in [t,\tau]
      \\
      V^{t,\bx;\tau}_s(e)\geq 0,\quad d\Prob\otimes ds\otimes \lambda(de)-\text{a.e.}
    \end{cases}
  \end{align}
  In \eqref{ekv:bsde-c-jmp}, the state pair $(I^t,X^{t,\bx})$ satisfy the forward SDE
  \begin{align*}
  \begin{cases}
    I^t_s=\beta_0+\int_t^s\!\!\int_U(e-I^t_{r-})\mu(dr,de),\quad\forall s\in [t,T],\\
    X^{t,\bx}_s=\bx(s\wedge t)+\int_t^{s\vee t} a(r,X^{t,\bx},I^t_r)dr+\int_t^{s\vee t}\sigma(r,X^{t,\bx},I^t_r)dW_r,\quad\forall s\in [0,T],
  \end{cases}
  \end{align*}
  where $\beta_0 \in U$ is the (arbitrary) initial value of $I^t$ at time $t$.
\end{thm}

\medskip

The above theorem extends earlier results, in particular those of~\cite{Choukroun15}, by establishing a direct link between BSDEs with constrained jumps and zero-sum controller-stopper games. This connection allows for the representation of stochastic differential games driven by path-dependent SDEs, thereby generalizing the Markovian framework considered in~\cite{Choukroun15}. Moreover, even in the Markovian setting, our approach via non-linear Snell envelopes, enables us to relax some of the strict regularity assumptions on the barrier $\psi$ imposed in~\cite{Choukroun15}.

\subsection{The dual game}
The non-linear Snell envelope in Theorem~\ref{thm:zs-game} was analyzed in \cite{imp-stop-game}. Here we recall the key results from \cite{imp-stop-game} and reformulate the associated dual game in a form tailored to our setting. Consider the sequence $(Y^{t,\bx,n},Z^{t,\bx,n},V^{t,\bx,n},K^{+,t,\bx,n})\in\mcS^2_{t,\tau}\times\mcH^2_{t}(W)\times \mcH^2_{t}(\mu)\times \mcA^2_{t}$ defined for each $n\in\bbN$ as the unique solution to the reflected BSDEs that penalize negative jumps,
\begin{align}\label{ekv:rbsde-pen}
\begin{cases}
  Y^{t,\bx,n}_s=\psi(T,X^{t,\bx})+\int_s^T f(r,X^{t,\bx},I^t_r)dr-n\int_s^T\!\!\int_U(V^{t,\bx,n}_r(e))^-\lambda(de)dr-\int_s^T Z^{t,\bx,n}_r dW_r
  \\
  \quad-\int_s^T\!\!\int_U V^{t,\bx,n}_r(e)\mu(dr,de)+K^{+,t,\bx,n}_T-K^{+,t,\bx,n}_s,\quad\forall s\in [t,T]
  \\
  Y^{t,\bx,n}_s\geq \psi(s,X^{t,\bx}),\:\forall s\in [t,T]\quad\text{and}\quad \int_t^T(Y^{t,\bx,n}_s - \psi(s,X^{t,\bx}))d K^{+,t,\bx,n}_s=0.
\end{cases}
\end{align}
It is well known (see \eg \cite{QuenSul14}) that $\tau_n:=\inf\{s\geq t:Y^{t,\bx,n}_s=\psi(s,X^{t,\bx})\}\in\mcT^\mcR_t$ is an optimal stopping time for the corresponding optimal stopping problem. More generally, we let $\mcV$ be the set of all $\Pred(\bbF^\mcR)\otimes\mcB(U)$-measurable bounded maps $\nu=\nu_t(\omega,e):[0,T]\times\Omega\times U\to [0,\infty)$ and define for each $(t,\bx)\in[0,T]\times \bfC^d$ and $(\nu,\tau)\in \mcV\times\mcT^\mcR_t$, the corresponding cost/reward as
\begin{align*}
  J^{\mcR}(t,\bx;\nu,\tau)&:=\E^\nu\Big[\psi(\tau,X^{t,\bx})+\int_t^{\tau}f(s,X^{t,\bx},I^t_s)ds\,\Big|\,\mcF_t\Big].
\end{align*}
In the definition of the randomized reward/cost functional $J^\mcR$ above, $\E^\nu$ is expectation with respect to the probability measure $\Prob^\nu$ on $(\Omega,\mcF)$ defined by $d\Prob^\nu:=\kappa^\nu_T d\Prob$ with
\begin{align*}
\kappa^{\nu}_s&:=\mcE_s\Big(\int_{t}^\cdot\!\!\int_U(\nu_r(e)-1)(\mu(dr,de)-\lambda(de)dr)\Big)
\\
&:=\exp\Big(\int_{t}^s\!\!\!\int_U(1-\nu_r(e))\lambda(de)dr\Big)\prod_{t<\sigma_j\leq s}\nu_{\sigma_j}(\zeta_j),
\end{align*}
where $(\sigma_j,\zeta_j)_{j\in\bbN}$ denote the jump times and marks of $\mu$. We then have the following:

\begin{lem}\label{lem-3_4-tilde}
For each $(t,\bx)\in[0,T]\times \bfC^d$ and $n\in\bbN$, the dual game admits a value, \ie
\begin{align}\label{ekv:dual-value-n}
  Y^{t,\bx,n}_t = \essinf_{\nu\in\mcV^n}\esssup_{\tau\in\mcT^\mcR_t}J^{\mcR}(t,\bx;\nu,\tau)= \esssup_{\tau\in\mcT^\mcR_t}\essinf_{\nu\in\mcV^n}J^{\mcR}(t,\bx;\nu,\tau),
\end{align}
where $\mcV^n$ is the subset of $\mcV$ with all maps $\nu:[0,T]\times\Omega\times U\to [0,n]$. Moreover, letting $\nu^n_s(e):=n\ett_{[V^{t,\bx,n}_s(e)<0]}$, the pair $(\nu^n,\tau_n)$ is a saddle-point for the game in the sense that
\begin{align}\label{ekv:dual-saddle-trunk}
  J^{\mcR}(t,\bx;\nu^n,\tau)\leq J^{\mcR}(t,\bx;\nu^n,\tau_n)\leq J^{\mcR}(t,\bx;\nu,\tau_n)
\end{align}
for all $(\nu,\tau)\in\mcV^n\times\mcT^\mcR_t$.
\end{lem}

\begin{proof}
Lemma 3.4 in \cite{imp-stop-game} covers the corresponding result when the randomized reward/cost functional is replaced by
\begin{align*}
  \bar J^{\mcR}(t,\bx;\nu,\tau)&:=\E^\nu\Big[\psi(\tau,X^{t,\bx})+\int_t^{\tau}f(s,X^{t,\bx},I^t_s)ds\,\Big|\,\mcF^\mcR_t\Big]
\end{align*}
allowing us to deduce that $Y^{t,\bx,n}_t=\bar J^{\mcR}(t,\bx;\nu^n,\tau_n)$ and that
\begin{align}\label{ekv:ext-dual-saddle}
  \bar J^{\mcR}(t,\bx;\nu^n,\tau)\leq \bar J^{\mcR}(t,\bx;\nu^n,\tau_n)\leq \bar J^{\mcR}(t,\bx;\nu,\tau_n)
\end{align}
for all $(\nu,\tau)\in\mcV^n\times\mcT^\mcR_t$. On the other hand, standard results for BSDEs imply that $Y^{t,\bx,n}_t$ is $\Prob$-a.s.~constant and thus,
\begin{align*}
  J^{\mcR}(t,\bx;\nu^n,\tau_n)=\bar J^{\mcR}(t,\bx;\nu^n,\tau_n),\quad\Prob-\text{a.s.}
\end{align*}
Now, $\mcF_t\subset \mcF^\mcR_t$ and taking the conditional expectation on both sides of the first inequality in \eqref{ekv:ext-dual-saddle} and using the tower property yields that
\begin{align*}
  J^{\mcR}(t,\bx;\nu^n,\tau)=\E^{\nu^n}\big[\bar J^{\mcR}(t,\bx;\nu^n,\tau)\,\big|\,\mcF_t\big]\leq \E^{\nu^n}\big[\bar J^{\mcR}(t,\bx;\nu^n,\tau_n)\,\big|\,\mcF_t\big]=J^{\mcR}(t,\bx;\nu^n,\tau_n).
\end{align*}
In the same way, taking conditional expectations on both sides of the latter inequality in \eqref{ekv:ext-dual-saddle}, we obtain
\begin{align*}
  J^{\mcR}(t,\bx;\nu^n,\tau_n)\leq \E^{\nu}\big[\bar J^{\mcR}(t,\bx;\nu,\tau_n)\,\big|\,\mcF_t\big]=J^{\mcR}(t,\bx;\nu,\tau_n)
\end{align*}
proving \eqref{ekv:dual-saddle-trunk} from which \eqref{ekv:dual-value-n} follows.
\end{proof}

Since the game in \eqref{ekv:dual-value-n} is defined over stopping times with respect to the filtration generated by $W$ and $\mu$, and the control is represented by the point process $I^t$ driven by $\mu$, non-anticipativity is built into the formulation. This yields a symmetric setup and is what allows us to extract the saddle-point $(\nu^n,\tau_n)$.

Now, the sequence of processes $(Y^{t,\bx,n})_{n\in\bbN}$ is non-increasing and bounded from below by the process $\psi(\cdot,X^{t,\bx})$, implying the existence of a $\Prog(\bbF^{\mcR})$-measurable process $Y^{t,\bx}$ such that $Y^{t,\bx,n}\searrow Y^{t,\bx}$ pointwise, $\Prob$-a.s.~as $n\to\infty$. In \cite{imp-stop-game} it was shown that $Y^{t,\bx}\in\mcS^2_t$ satisfies
\begin{align*}
  Y^{t,\bx}_s=\esssup_{\tau\in\mcT^\mcR_t}Y^{t,\bx;\tau}_s
\end{align*}
and that the stopping time
\begin{align}\label{ekv:opt-stop-dual}
  \tau^{\mcR}:=\inf\{r\geq t: Y^{t,\bx}_r=\psi(r,X^{t,\bx})\}
\end{align}
is optimal in the sense that $Y^{t,\bx}_s=Y^{t,\bx;\tau^{\mcR}}_s$ for all $s\in [t,\tau^\mcR]$.

Moreover, standard results such as flow properties of the related SDEs and stability estimates for the reflected BSDEs result in the existence of a continuous (deterministic) $\bbC$-progressively measurable map $v^{\mcR,n}:[0,T]\times\bfC^d\to\R$ for which $Y^{t,\bx,n}_t=v^{\mcR,n}(t,\bx)$, $\Prob$-a.s. Consequently, there is an upper semi-continuous (deterministic) $\bbC$-progressively measurable map $v^{\mcR}:[0,T]\times\bfC^d\to\R$ for which $Y^{t,\bx}_t=v^{\mcR}(t,\bx)$ and appealing to Theorem 3.1 in \cite{imp-stop-game} we find that
\begin{align*}
  v^{\mcR}(t,\bx) = \essinf_{\nu\in\mcV}\esssup_{\tau\in\mcT^\mcR_t}J^\mcR(t,\bx;\nu,\tau) = \essinf_{\nu\in\mcV}J^\mcR(t,\bx;\nu,\tau^\mcR).
\end{align*}

\section{Proof of Theorem~\ref{thm:zs-game}\label{sec:proof}}
This section is dedicated to finalising the proof of Theorem~\ref{thm:zs-game} by showing that the map $v^\mcR$ is continuous and satisfies $v^{\mcR}(t,\bx)=\underline v(t,\bx)=\bar v(t,\bx)$. To establish this result, we approximate the control sets by finite-valued discretizations.

\subsection{Discretization of the upper and lower value functions}
We fix $t\in [0,T]$ and introduce the following discretization:
\begin{defn}\label{def:discretization}
For each $\eps>0$:
\begin{itemize}
  \item We let $n^\eps\geq 0$ be the smallest integer such that $2^{-n^\eps}(T-t)\leq\eps$, set $n_{\bbT}^{t,\eps}:=2^{n^\eps}+1$ and introduce the discrete set $\bbT^{t,\eps}:=\{t^\eps_i:t^\eps_i=t+(i-1)2^{-n^\eps}(T-t),i=1,\ldots,n_{\bbT}^\eps\}$, a discretization of $[t,T]$ with step-size $\Delta^{t,\eps}:=2^{-n^\eps}(T-t)$. For $s\in [t,T]$, we let $\bbT^{t,\eps}_s:=\bbT^{t,\eps}\cap[s,T]$.
  \item We let $(U^\eps_{i})_{i=1}^{n^\eps_U}$ be a Borel-partition of $U$ such that each $U^\eps_i$ has non-empty interior in $U$ and a diameter that does not exceed $\eps$ and let $(b^\eps_i)_{i=1}^{n^\eps_U}$ be a sequence with $b^\eps_i\in \text{int}\,U^\eps_i$ and denote by $\bar U^\eps:=\{b^\eps_1,\ldots,b^\eps_{n^\eps_U}\}$ the corresponding discretization of $U$.
  \item We let $\Xi_\bbT^{t,\eps}(s):=\inf\{r\geq s:r\in\bbT^{t,\eps}\}$.
\end{itemize}
\end{defn}
Moreover, we define a corresponding discretization of the control set by letting
\begin{align*}
\mcU^{\eps}_t:=\Big\{u\in\mcU_t: \exists (\beta_i)_{i=1}^{n_{\bbT}^\eps-1},\,\beta_i:\Omega\mapsto \bar U^\eps\in m\mcF_{t^\eps_i},\, u_s= \sum_{i=1}^{n_{\bbT}^\eps-2}\beta_{i}\ett_{[t^\eps_i,t^\eps_{i+1})}(s)+\beta_{n_{\bbT}^\eps-1}\ett_{[t^\eps_{n_{\bbT}^\eps-1},T]}(s)\Big\}.
\end{align*}
We introduce the projection $\Xi_\mcU^\eps:\mcU_t\to\mcU^\eps_t$ defined for each $u\in\mcU_t$ as
\begin{align*}
  \Xi_\mcU^\eps[u](s):= \sum_{i=1}^{n_{\bbT}^\eps-2}b^\eps_{\iota^\eps_i(u)}\ett_{[t^\eps_i,t^\eps_{i+1})}(s) + b^\eps_{\iota^\eps_{n_{\bbT}^\eps-1}(u)}\ett_{[t^\eps_{n_{\bbT}^\eps-1},T]}(s),
\end{align*}
where $\iota^\eps_i(u)$ is a measurable selection of
\begin{align*}
  \iota^\eps_i(u)\in\argmin_{j\in\{1,\ldots,n^\eps_U\}} \text{dist}\Big(U^\eps_j,\frac{1}{\Delta^{t,\eps}}\E\Big[\int_{t^\eps_i}^{t^\eps_{i+1}}u_sds\,\Big|\,\mcF_{t^\eps_i}\Big]\Big),
\end{align*}
with
\begin{align*}
  \text{dist}(A,x):=\inf\{|y-x|:y\in A\}.
\end{align*}
The following result is classical and follows from density of the set of piecewise constant adapted processes in the set of progressively measurable processes under the $L^1$-norm and the compactness of $U$.
\begin{lem}\label{lem:dens}
  For any $u\in\mcU_t$, we have that $\rho_t(u,\Xi_\mcU^\eps(u))\to 0$ as $\eps\to 0$.
\end{lem}
In addition to the discretization of $\mcU_t$, we let $\mcT^{\eps}_t$ be the subset of $\mcT_t$ with all stopping times $\tau$ for which $\tau\in\bbT^{t,\eps}$, $\Prob$-a.s., and define the corresponding non-anticipative strategies as
\begin{defn}
For $\eps>0$, let $\mcT^{S,\eps}_t$ be the subset of $\mcT^{S}_t$ containing all strategies $\tau^S:\mcU_t\to\mcT^{\eps}_t$ such that for any $u,\tilde u\in \mcU_t$, the sets $\{\omega:\tau^S(u) = t^\eps_i\}$ and $\{\omega:\tau^S(\tilde u)= t^\eps_i\}$ differ by at most a $\Prob$-null set, whenever $\Xi_\mcU^\eps[u]=\Xi_\mcU^\eps[\tilde u]$, $dr\otimes d\Prob$-a.e.~on $[t,t^\eps_i]\times\Omega$ for $i=1,\ldots,n_\bbT^{t,\eps}$.
\end{defn}

We then introduce the lower value function with control and stopping discretization as follows
\begin{align}\label{ekv:lvf-trunk-disc}
\underline v^{\eps}(t,\bx):=\essinf_{u\in\mcU^{\eps}_t}\,\esssup_{\tau\in\mcT^\eps_t}J(t,\bx;u,\tau).
\end{align}
%and the corresponding upper value function
%\begin{align}\label{ekv:uvf-trunk-disc}
%\bar v^{\eps}(t,\bx):=\esssup_{\tau^S\in\mcT^{S,\eps}_t}\,\essinf_{u\in\mcU^{\eps}_t}J(t,\bx;u,\tau^S(u)).
%\end{align}

The following lemma is central in proving that $\underline v^{\eps}$ approximates $\underline v$ as $\eps\to 0$:
\begin{lem}\label{lem:J-cont}
  For any sequences
  \begin{itemize}
    \item $(t_i,\bx_i)_{i\in\bbN}\subset[0,T]\times\bfC^d$ and $(\tilde t_i,\tilde \bx_i)_{i\in\bbN}\subset[0,T]\times\bfC^d$ that are bounded with $\textbf d_\Lambda[(t_i,\bx_i),(\tilde t_i,\tilde\bx_i)]\to 0$,
    \item $(u^i)_{i\in\bbN}$ and $(\tilde u^i)_{i\in\bbN}$, with $u^i\in\mcU_{t_i}$, $\tilde u^i\in\mcU_{\tilde t_i}$ and $\rho_{t_i\vee\tilde t_i}(u^i,\tilde u^i)\to 0$; and
    \item $(\tau_i)_{i\in\bbN}$ and $(\tilde \tau_i)_{i\in\bbN}$ with $\tau_i\in\mcT_{t_i}$, $\tilde\tau_i\in\mcT_{\tilde t_i}$ $\tau_i\leq \tilde \tau_i$ and $\tilde\tau_i-\tau_i\to 0$, $\Prob$-a.s.~as $i\to\infty$,
  \end{itemize}
  we have
  \begin{align}\label{ekv:seq-to-0}
    \limsup_{i\to\infty}\E\Big[\psi(\tau_i,X^{t_i,\bx_i;u^i})-\psi(\tilde\tau_i,X^{\tilde t_i,\tilde\bx_i;\tilde u^i}) +\int_{t_i}^{\tau}f(s,X^{t_i,\bx_i;u^i},u^i_s)ds-\int_{\tilde t_i}^{\tilde\tau_i}f(s,X^{\tilde t_i,\tilde \bx_i;\tilde u^i},\tilde u^i_s)ds\Big]\leq 0.
  \end{align}
\end{lem}

\begin{proof}[Proof.]
Let us denote the right-hand side of \eqref{ekv:seq-to-0} by $\zeta_i$ and note that by polynomial growth, \eqref{ekv:Xu-growth} and Markov's inequality, there is for each $\delta>0$ a $K>0$ such that
\begin{align*}
  \zeta_i&\leq\E\Big[\ett_{[\|X^{t_i,\bx_i;u^i}\|_{T}\vee \|X^{\tilde t_i,\tilde\bx_i;\tilde u^i}\|_{T}\leq K]}\Big(\psi(\tau_i,X^{t_i,\bx_i;u^i})-\psi(\tilde\tau_i,X^{\tilde t_i,\tilde\bx_i;\tilde u^i})
  \\
  &\quad+\int_{t_i}^{\tau}f(s,X^{t_i,\bx_i;u^i},u^i_s)ds-\int_{\tilde t_i}^{\tilde\tau_i}f(s,X^{\tilde t_i,\tilde \bx_i;\tilde u^i},\tilde u^i_s)ds\Big)\Big]+\delta.
\end{align*}
Using local uniform continuity, polynomial growth and \eqref{ekv:Xu-growth} again we find that there is a modulus of continuity $\rho$ such that
\begin{align*}
  \zeta_i&\leq\E\big[\rho(\mathbf d_\Lambda[(\tau_i,X^{t_i,\bx_i;u^i}),(\tilde \tau_i,X^{\tilde t_i,\tilde\bx_i;\tilde u^i})]\wedge (2K+T))\big]+C(|t_i-\tilde t_i| +\E\big[|\tau_i-\tilde \tau_i|^2\big]^{1/2})+\delta.
\end{align*}
On the other hand, \eqref{ekv:Xu-stab} and the fact that $\tau_i-\tilde \tau_i\to 0$, $\Prob$-a.s.~as $i\to\infty$ gives that\\ $d_\Lambda[(\tau_i,X^{t_i,\bx_i;u^i}),(\tilde \tau_i,X^{\tilde t_i,\tilde\bx_i;\tilde u^i})]\to 0$ in probability. Consequently, the right-hand side tends to $\delta$ as $i\to\infty$ and the result follows since $\delta>0$ was arbitrary.
\end{proof}

\medskip

\begin{lem}\label{lem:underv-eps-conv}
For each $(t,\bx)\in[0,T]\times\bfC^d$, we have $\underline v^{\eps}(t,\bx)\to \underline v(t,\bx)$ in $L^1(\Omega,\mcF_t,\Prob)$ as $\eps\to 0$.
\end{lem}

\noindent\emph{Proof.} We have
\begin{align*}
\underline v(t,\bx)-\underline v^{\eps}(t,\bx)&\leq\essinf_{u\in\mcU^{\eps}_t}\,\esssup_{\tau\in\mcT_t}J(t,\bx;u,\tau) - \essinf_{u\in\mcU^{\eps}_t}\esssup_{\tau\in\mcT^{\eps}_t}J(t,\bx;u,\tau)
\\
&\leq \esssup_{u\in\mcU^{\eps}_t}\,\esssup_{\tau\in\mcT_t}\{J(t,\bx;u,\tau)-J(t,\bx;u,\Xi_\bbT^{t,\eps}(\tau))\}
\end{align*}
Moreover, for each $u\in\mcU^{\eps}_t$ and $\tau\in\mcT_t$, we have
\begin{align*}
J(t,\bx;u,\tau)-J(t,\bx;u,\Xi_\bbT^{t,\eps}(\tau))&= \E\Big[\psi(\tau,X^{t,\bx;u})-\psi(\Xi_\bbT^{t,\eps}(\tau),X^{t,\bx;u})-\int_{\tau}^{\Xi_\bbT^{t,\eps}(\tau)} f(r,X^{t,\bx;u},u_r)dr\Big|\mcF_t\Big]
\end{align*}
and since $\Xi_\bbT^{t,\eps}(\tau)\searrow \tau$, $\Prob$-a.s.,~as $\eps\to 0$, it follows by Lemma~\ref{lem:J-cont} that $\E\big[(\underline v(t,\bx)-\underline v^\eps(t,\bx))^+\big]\to 0$ as $\eps\to 0$.

\medskip

We approach the opposite inequality by noting that for each $\delta>0$, there is a $u^\delta\in\mcU$ (independent of $\eps$) such that
\begin{align*}
\underline v^{\eps}(t,\bx) - \underline v(t,\bx) &\leq \essinf_{u\in\mcU^{\eps}_t}\,\esssup_{\tau\in\mcT^\eps_t}J(t,\bx;u,\tau) - \essinf_{u\in\mcU_t}\esssup_{\tau\in\mcT^{\eps}_t}J(t,\bx;u,\tau)
\\
&\leq \essinf_{u\in\mcU^{\eps}_t}\,\esssup_{\tau\in\mcT^\eps_t}J(t,\bx;u,\tau) - \esssup_{\tau\in\mcT^{\eps}_t}J(t,\bx;u^\delta,\tau)+\delta
\\
&\leq \esssup_{\tau\in\mcT^{\eps}_t}\{J(t,\bx;\Xi_\mcU^\eps(u^\delta),\tau) - J(t,\bx;u^\delta,\tau)\}+\delta.
\end{align*}
Taking the limit as $\eps\to 0$ and using Lemma~\ref{lem:dens} and Lemma~\ref{lem:J-cont}, we find that $\limsup_{\eps\to 0}\E\big[(\underline v(t,\bx)-\underline v^\eps(t,\bx))^+\big]\leq \delta$ and the assertion follows since $\delta>0$ was arbitrary.\qed\\

We also use discretization to justify the terminology ``upper'' and ``lower'' value functions by establishing the following inequality.

\begin{prop}\label{prop:vf-just}
For each $(t,\bx)\in [0,T]\times \bfC^d$, we have $\underline v(t,\bx) \leq \bar v(t,\bx)$, $\Prob$-a.s.
\end{prop}

To prove Proposition~\ref{prop:vf-just} we utilize a lower value function where discretization only applies to the stopping time and thus introduce the value function
\begin{align}\label{ekv:lvf-stop-disc}
\underline v^t_{\eps}(s,\bx):=\essinf_{u\in\mcU_{s}}\,\esssup_{\tau\in\mcT^{t,\eps}_s}J(s,\bx;u,\tau),\quad\forall (s,\bx)\in[t,T]\times \bfC^d,
\end{align}
where $\mcT^{t,\eps}_s$ is the set of all $\tau\in\mcT^\eps_t$ such that $\tau\geq s$, $\Prob$-a.s.

Similarly to the above we find that
\begin{align*}
\underline v(t,\bx)-\underline v^t_{\eps}(t,\bx)&\leq \esssup_{u\in\mcU_t}\,\esssup_{\tau\in\mcT_t}\{J(t,\bx;u,\tau)-J(t,\bx;u,\Xi_\bbT^{t,\eps}(\tau))\}
\end{align*}
and it follows by Lemma~\ref{lem:J-cont} that $\limsup_{\eps\to 0}\E\big[(\underline v(t,\bx)-\underline v^t_{\eps}(t,\bx))^+\big]\leq 0$. On the other hand, we have the following:

\begin{lem}\label{lem:DynP-eps}
The map $\underline v^{t}_{\eps}:[t,T]\times\bfC^d\to\R$ satisfies the dynamic programming relation
\begin{align}\label{ekv:dynP-eps}
  \begin{cases}
    \underline v^{t}_{\eps}(T,\bx)=\psi(T,\bx),\quad\forall \bx\in\bfC^{d},
    \\
    \underline v^{t}_{\eps}(t^\eps_i,\bx)
    =\psi(t^\eps_i,\bx)\vee\essinf_{u\in\mcU_{t^\eps_i}}\E\big[\int_{t^\eps_i}^{t^\eps_{i+1}}f(s,X^{t^\eps_i,\bx;u},u_s)ds+\underline v^{t}_{\eps}(t^\eps_{i+1},X^{t^\eps_{i},\bx;u})\,\big|\,\mcF_{t^\eps_i}\big],
    \\
    \quad\forall (i,\bx)\in \{1,\ldots,n^{t,\eps}_\bbT-1\}\times\bfC^d.
  \end{cases}
\end{align}
\end{lem}

\begin{proof}
Let $w:\bbT^{t,\eps}\times\bfC^d\to\R$ be the unique solution to the recursion \eqref{ekv:dynP-eps} and note that for each $i\in\{1,\ldots,n_\bbT^{t,\eps}\}$, the map $\bx\mapsto w(t^\eps_i,\bx):\bfC^d\to \R$ is continuous. The proof is based on induction and we assume that $w(t^\eps_{i},\bx) = \underline v^{t}_{\eps}(t^\eps_{i},\bx)$ for all $(i,\bx)\in \{i'+1,i'+2,\ldots,n^{t,\eps}_\bbT\}\times\bfC^d$ for some $i'\in\{1,2,\ldots,n^{t,\eps}_\bbT-1\}$. We pick arbitrary $\bx\in \bfC^d$ and prove that $w(t^\eps_{i'},\bx)=\underline v^{t}_{\eps}(t^\eps_{i'},\bx)$ over the two steps below.

\medskip

\textbf{Step 1:} We first prove that $w(t^\eps_{i'},\bx)\leq \underline v^{t}_{\eps}(t^\eps_{i'},\bx)$.
We pick $\delta>0$ and an arbitrary $u\in\mcU_{t^\eps_{i'}}$ and have that
\begin{align*}
  w(t^\eps_{i'},\bx)&\leq \esssup_{\tau\in\mcT^{t,\eps}_{t^\eps_{i'}}} \E\Big[\ett_{[\tau=t^\eps_{i'}]}\psi(t^\eps_{i'},\bx) + \ett_{[\tau>t^\eps_{i'}]}\Big(\int_{t^\eps_{i'}}^{t^\eps_{i'+1}}f(s,X^{t^\eps_{i'},\bx;u},u_s)ds+ w(t^\eps_{i'+1},X^{t^\eps_{i'},\bx;u})\Big)\,\Big|\,\mcF_{t^\eps_{i'}}\Big]
  \\
  &\leq \E\Big[\ett_{[\tau^\delta_1=t^\eps_{i'}]}\psi(t^\eps_{i'},\bx) + \ett_{[\tau^\delta_1>t^\eps_{i'}]}\Big(\int_{t^\eps_{i'}}^{t^\eps_{i'+1}}f(s,X^{t^\eps_{i'},\bx;u},u_s)ds+ w(t^\eps_{i'+1},X^{t^\eps_{i'},\bx;u})\Big)\,\Big|\,\mcF_{t^\eps_{i'}}\Big]+\delta
\end{align*}
for some $\tau^\delta_1\in \mcT^{t,\eps}_{t^\eps_{i'}}$. Since $\bx\mapsto w(t^\eps_{i'+1},\bx)$ and the coefficients $f$ and $\psi$ are continuous and of polynomial growth, and $\bfC^d$ is separable, there is a sequence of sets $(\mcO^\delta_l)_{l\in\bbN}$, with $\mcO^\delta_l\subset \mcC_t$ and a corresponding sequence $(\bx_l)_{l\in\bbN}$, with $\bx_l\in\mcO^\eps_l$, such that $|w(t^\eps_{i'+1},\bx_l)-w(t^\eps_{i'+1},\tilde\bx)|\leq \delta$ and
\begin{align*}
  \E\big[|J(t^\eps_{i'+1},\bx_l;\tau,u)-J(t^\eps_{i'+1},\tilde \bx;\tau,u)|\,\big|\,\mcF_{t^\eps_{i'}}\big]\leq \delta,\quad\Prob\text{-a.s.}
\end{align*}
for all $\tilde\bx\in\mcO^\delta_l$ and $(\tau,u)\in\mcT_t\times\mcU_t$. To obtain the latter inequality on the sets $\mcO^\delta_l$, we use that the map $\tilde \bx\mapsto J(t^\eps_{i'+1},\tilde\bx;\tau,u)$ is continuous, uniformly in $(\tau,u)\in\mcT_{t^\eps_{i'+1}}\times\mcU_{t^\eps_{i'+1}}$. Showing this can be done by following along the lines of the proof of Lemma~\ref{lem:J-cont}.

Now, by our induction assumption
\begin{align*}
  w(t^\eps_{i'+1},\bx_l)&\leq \esssup_{\tau\in\mcT^{t,\eps}_{t^\eps_{i'+1}}} J(t^\eps_{i'+1},\bx_l;u,\tau)
\end{align*}
and consequently there is a sequence $(\tau^\delta_{2,l})_{l\in\bbN}$ in $\mcT^{t,\eps}_{t^\eps_{i'+1}}$ such that
\begin{align*}
  w(t^\eps_{i'+1},\bx_l)&\leq J(t^\eps_{i'+1},\bx_l;u,\tau^\delta_{2,l})+\delta,\quad \forall l\in\bbN.
\end{align*}
Letting
\begin{align*}
  \tau^\delta:=t^\eps_{i'+1}\ett_{[\tau^\delta_1=t^\eps_{i'}]} + \ett_{[\tau^\delta_1>t^\eps_{i'}]}\sum_{l\in\bbN}\ett_{\mcO^\delta_l}(X^{t^\eps_{i'},\bx;u})\tau^\delta_{2,l}
\end{align*}
we get that
\begin{align*}
  w(t^\eps_{i'},\bx)&\leq \E\Big[\ett_{[\tau^\delta_1=t^\eps_{i'}]}\psi(t^\eps_{i'},\bx) + \ett_{[\tau^\delta_1>t^\eps_{i'}]}\Big(\int_{t^\eps_{i'}}^{t^\eps_{i'+1}}f(s,X^{t^\eps_{i'},\bx;u},u_s)ds
  \\
  &\quad + \sum_{l\in\bbN}\ett_{X^{t^\eps_{i'},\bx;u}\in\mcO^\delta_l}w(t^\eps_{i'+1},\bx_l)\Big)\,\Big|\,\mcF_{t^\eps_{i'}}\Big]+2\delta
  \\
  &\leq \E\Big[\ett_{[\tau^\delta_1=t^\eps_{i'}]}\psi(t^\eps_{i'},\bx) + \ett_{[\tau^\delta_1>t^\eps_{i'}]}\Big(\int_{t^\eps_{i'}}^{t^\eps_{i'+1}}f(s,X^{t^\eps_{i'},\bx;u},u_s)ds
  \\
  &\quad + \sum_{l\in\bbN}\ett_{X^{t^\eps_{i'},\bx;u}\in\mcO^\delta_l}J(t^\eps_{i'+1},\bx_l;u,\tau^\delta_{2,l})\,\Big|\,\mcF_{t^\eps_{i'}}\Big]+3\delta
  \\
  &\leq \E\Big[\ett_{[\tau^\delta_1=t^\eps_{i'}]}\psi(t^\eps_{i'},\bx) + \ett_{[\tau^\delta_1>t^\eps_{i'}]}\Big(\int_{t^\eps_{i'}}^{t^\eps_{i'+1}}f(s,X^{t^\eps_{i'},\bx;u},u_s)ds
  \\
  &\quad + \sum_{l\in\bbN}\ett_{X^{t^\eps_{i'},\bx;u}\in\mcO^\delta_l}J(t^\eps_{i'+1},X^{t^\eps_{i'},\bx;u},\tau^\delta_{2,l})\,\Big|\,\mcF_{t^\eps_{i'}}\Big]+4\delta
  \\
  &=J(t^\eps_{i'},\bx;u,\tau^\delta)+4\delta.
\end{align*}
Hence,
\begin{align*}
 w(t^\eps_{i'},\bx)\leq \esssup_{\tau\in\mcT^{t,\eps}_s}J(t^\eps_{i'},\bx;u,\tau)+4\delta.
\end{align*}
Since $\delta>0$ and $u$ were arbitrary it follows that $w(t^\eps_{i'},\bx)\leq \underline v^{t}_{\eps}(t^\eps_{i'},\bx)$.\\

\textbf{Step 2:} We show that $w(t^\eps_{i'},\bx)\geq \underline v^{t}_{\eps}(t^\eps_{i'},\bx)$. For $\delta>0$ we thus pick $u^{\delta,1}\in\mcU_{t^\eps_{i'}}$ such that
\begin{align*}
  w(t^\eps_{i'},\bx)&\geq \E\Big[\ett_{[\tau=t^\eps_{i'}]}\psi(t^\eps_{i'},\bx) + \ett_{[\tau>t^\eps_{i'}]}\Big(\int_{t^\eps_{i'}}^{t^\eps_{i'+1}}f(s,X^{t^\eps_{i'},\bx;u^{\delta,1}},u^{\delta,1}_s)ds+ w(t^\eps_{i'+1},X^{t^\eps_{i'},\bx;u^{\delta,1}})\Big)\,\Big|\,\mcF_{t^\eps_{i'}}\Big]-\delta
\end{align*}
for any $\tau\in\mcT^{t,\eps}_{t^{\eps}_{i'}}$. On the other hand, there is a sequence $(u^{\delta,2,l})_{l\in\bbN}$ in $\mcU_{t^\eps_{i'+1}}$ such that
\begin{align*}
  w(t^\eps_{i'+1},\bx_l)\geq J(t^\eps_{i'+1},\bx_l;u^{\delta,2,l};\tau\vee t^\eps_{i'+1})-\delta,\quad\forall l\in\bbN.
\end{align*}
By letting
\begin{align*}
  u^\delta_s:=\ett_{[s< t^\eps_{i'+1}]}u^{\delta,1}_s+\ett_{[s \geq t^\eps_{i'+1}]}\sum_{l\in\bbN}\ett_{\mcO^\delta_l}(X^{t^\eps_{i'},\bx;u^{\delta,1}})u^{\delta,2,l}_s
\end{align*}
we thus find that
\begin{align*}
 w(t^\eps_{i'},\bx)&\geq J(t^\eps_{i'},\bx;u^\delta,\tau)+4\delta.
\end{align*}
Taking the supremum over all $\tau\in\mcT^{t,\eps}_{t^\eps_{i'}}$ thus gives that
\begin{align*}
 w(t^\eps_{i'},\bx)&\geq \esssup_{\tau\in\mcT^{t,\eps}_{t^\eps_{i'}}}J(t^\eps_{i'},\bx;u^\delta,\tau)+4\delta\geq \underline v^{t}_{\eps}(t^\eps_{i'},\bx)+4\delta
\end{align*}
and the assertion follows as $\delta>0$ was arbitrary.
\end{proof}

\bigskip

\begin{lem}
Define the non-anticipative stopping strategy $\tau^{S}_\eps(u):\mcU_t\to\mcT^\eps_t$ as $\tau^{S}_\eps(u):=\inf\{s\in\bbT^{t,\eps}: \underline v^{t}_\eps(s,X^{t,\bx;u})=\psi(s,X^{t,\bx;u})\}$, then
\begin{align*}
  \underline v^{t}_\eps(t,\bx)=\essinf_{u\in\mcU_t}J(t,\bx;u,\tau^{S}_\eps(u)).
\end{align*}
\end{lem}

\begin{proof}
For $(s,\by)\in \bbT^{t,\eps}\times\bfC^d$ let $\tau^{S}_{\eps,s,\by}(u):=\inf\{r\in\bbT^{t,\eps}_s: \underline v^{t}_\eps(r,X^{s,\by;u})=\psi(r,X^{s,\by;u})\}$ and assume that for some $i'\in \{1,\ldots,n^{t,\eps}_\bbT-1\}$, we have
\begin{align}\label{ekv:induc-ass}
  \underline v^{t}_\eps(t^{\eps}_{i'+1},\by)=\essinf_{u\in\mcU_{t^{\eps}_{i'+1}}}J(t^{\eps}_{i'+1},\by;u,\tau^{S}_{\eps,t^{\eps}_{i'+1},\by}(u)), \quad\forall \by\in\bfC^d.
\end{align}
The recursion in \eqref{ekv:dynP-eps} then implies that for any $\by\in\bfC^d$ and any $\tilde u\in\mcU_{t^\eps_{i'}}$,
\begin{align*}
  \underline v^{t}_{\eps}(t^\eps_{i'},\by)
    &=\psi(t^\eps_{i'},\by)\vee\essinf_{u\in\mcU_{t^\eps_{i'}}}\E\Big[\int_{t^\eps_{i'}}^{t^\eps_{{i'}+1}}f(s,X^{t^\eps_{i'},\by;u},u_s)ds+\underline v^{t}_{\eps}(t^\eps_{i'+1},X^{t^\eps_{{i'}},\by;u})\,\Big|\,\mcF_{t^\eps_{i'}}\Big]
    \\
    &=\essinf_{u\in\mcU_{t^\eps_{i'}}}\E\Big[\ett_{[\tau^{S}_{\eps,t^\eps_{i'},\by}(u)=t^\eps_{i'}]}\psi(t^\eps_{i'},\by) + \ett_{[\tau^{S}_{\eps,t^\eps_{i'},\by}(u)>t^\eps_{i'}]}\Big(\int_{t^\eps_{i'}}^{t^\eps_{{i'}+1}}f(s,X^{t^\eps_{i'},\by;u},u_s)ds
    \\
    &\quad+\underline v^{t}_{\eps}(t^\eps_{i'+1},X^{t^\eps_{{i'}},\by;u})\Big)\,\Big|\,\mcF_{t^\eps_{i'}}\Big]
    \\
    &\leq \E\Big[\ett_{[\tau^{S}_{\eps,t^\eps_{i'},\by}(\tilde u)=t^\eps_{i'}]}\psi(t^\eps_{i'},\by) + \ett_{[\tau^{S}_{\eps,t^\eps_{i'},\by}(\tilde u)>t^\eps_{i'}]}\Big(\int_{t^\eps_{i'}}^{t^\eps_{{i'}+1}}f(s,X^{t^\eps_{i'},\by;\tilde u},\tilde u_s)ds
    \\
    &\quad+\underline v^{t}_{\eps}(t^\eps_{i'+1},X^{t^\eps_{{i'}},\by;\tilde u})\Big)\,\Big|\,\mcF_{t^\eps_{i'}}\Big].
\end{align*}
On the other hand, $\tau^{S}_{\eps,t^{\eps}_{i'},\by}(\tilde u)=\tau^{S}_{\eps,t^{\eps}_{i'+1},X^{t^\eps_{{i'}}}(\tilde u)}$, on the set $\{\tau^{S}_{\eps,t^{\eps}_{i'},\by}(\tilde u)>t^\eps_{i'}\}\setminus\mcN$ for some $\Prob$-null set $\mcN$. Hence,
\begin{align*}
  \ett_{[\tau^{S}_{\eps,t^\eps_{i'},\by}(\tilde u)>t^\eps_{i'}]}\underline v^{t}_{\eps}(t^\eps_{i'+1},X^{t^\eps_{{i'}},\by;\tilde u})\leq \ett_{[\tau^{S}_{\eps,t^\eps_{i'},\by}(\tilde u)>t^\eps_{i'}]}J(t^\eps_{i'+1},X^{t^\eps_{{i'}},\by;\tilde u};\tilde u,\tau^{S}_{\eps,t^{\eps}_{i'},\by}(\tilde u)),
\end{align*}
$\Prob$-a.s., which, inserted in the above inequality, yields that
\begin{align*}
  \underline v^{t}_{\eps}(t^\eps_{i'},\by)&\leq \E\Big[\ett_{[\tau^{S}_{\eps,t^\eps_{i'},\by}(\tilde u)=t^\eps_{i'}]}\psi(t^\eps_{i'},\by) + \ett_{[\tau^{S}_{\eps,t^\eps_{i'},\by}(\tilde u)>t^\eps_{i'}]}\Big(\int_{t^\eps_{i'}}^{t^\eps_{{i'}+1}}f(s,X^{t^\eps_{i'},\by;\tilde u},\tilde u_s)ds
    \\
    &\quad+J(t^\eps_{i'+1},X^{t^\eps_{{i'}},\by;\tilde u};\tilde u,\tau^{S}_{\eps,t^{\eps}_{i'},\by}(\tilde u))\Big)\,\Big|\,\mcF_{t^\eps_{i'}}\Big].
\end{align*}
By the tower property and since $\tilde u$ was arbitrary, we find that
\begin{align*}
  \underline v^{t}_\eps(t^\eps_{i'},\by)\leq\essinf_{u\in\mcU_t}J(t^\eps_{i'},\by;u,\tau^{S}_{\eps,t^\eps_{i'},\by}(u)), \quad\forall \by\in\bfC^d.
\end{align*}
Since $\tau^{S}_{\eps,t^\eps_{i'},\by}(u)\in\mcT^{t,\eps}_{t^\eps_{i'}}$ for all $u\in\mcU_{t^\eps_{i'}}$, we also have
\begin{align*}
  J(t^\eps_{i'},\by;u,\tau^{S}_{\eps,t^\eps_{i'},\by}(u)) \leq \esssup_{\tau\in\mcT^{\eps}_{t^\eps_{i'}}} J(t^\eps_{i'},\by;u,\tau),
\end{align*}
and thus
\begin{align*}
  \essinf_{u\in\mcU_t}J(t,\by;u,\tau^{S}_{\eps,t^\eps_{i'},\by}(u))
\leq \underline v^{t}_\eps(t,\by).
\end{align*}
In particular,
\begin{align*}
  \underline v^{t}_\eps(t^{\eps}_{i'},\by)=\essinf_{u\in\mcU_{t^{\eps}_{i'}}}J(t^{\eps}_{i'+1},\by;u,\tau^{S}_{\eps,t^{\eps}_{i'},\by}(u)), \quad\forall \by\in\bfC^d.
\end{align*}
Since \eqref{ekv:induc-ass} clearly holds for $i'=n^{t,\eps}_\bbT-1$, we can thus use induction to conclude that
\begin{align*}
  \underline v^{t}_\eps(t,\bx)&=\essinf_{u\in\mcU_{t}}J(t,\by;u,\tau^{S}_{\eps,t,\bx}(u))=\essinf_{u\in\mcU_{t}}J(t,\by;u,\tau^{S}_{\eps}(u)).
\end{align*}
Finally, it is straightforward to verify that $\tau^{S}_\eps$ is non-anticipative.
\end{proof}

\medskip

\begin{proof}[Proof of Proposition~\ref{prop:vf-just}]
By the previous lemma, we have that
\begin{align*}
\underline v^{t}_\eps(t,\bx)
= \essinf_{u\in\mcU_t}J(t,\bx;u,\tau^{S}_\eps(u)).
\end{align*}
Since $\tau^{S}_\eps$ is a non-anticipative stopping strategy and therefore belongs to $\mcT^S_t$, it follows that
\begin{align}\label{ekv:vf-sense-eps}
\underline v^{t}_\eps(t,\bx)
\leq \bar v(t,\bx).
\end{align}
Now,
\begin{align*}
  \limsup_{\eps\to 0}\E\big[(\underline v(t,\bx)-\underline v^t_{\eps}(t,\bx))^+\big]=0.
\end{align*}
Passing to the limit as $\eps\to 0$ in \eqref{ekv:vf-sense-eps} thus yields
\begin{align*}
\underline v(t,\bx) \leq \bar v(t,\bx),
\end{align*}
proving the proposition.
\end{proof}

\medskip

A consequence of Proposition~\ref{prop:vf-just} is that the proof of Theorem~\ref{thm:zs-game} naturally splits into two parts: establishing that $\bar v(t,\bx)\leq v^\mcR(t,\bx)$ and that $v^\mcR(t,\bx)\leq\underline v(t,\bx)$.
The first inequality follows directly from the results in Section~4.1 of \cite{Bandini18} (see also Section~5.2 of \cite{imp-stop-game}) once the second has been established.
We therefore focus on the latter inequality, which is more delicate and requires a sequence of careful estimates. More precisely, we prove that
\begin{align}\label{ekv:crand-dom}
v^\mcR(t,\bx)\leq\underline v(t,\bx),\quad \Prob\text{-a.s.},\quad \forall (t,\bx)\in [0,T]\times\bfC^d.
\end{align}

By Lemma~\ref{lem:underv-eps-conv}, it suffices to show that
\begin{align*}
\E\big[(v^{\mcR}(t,\bx)-\underline v^{\eps}(t,\bx))^+\big]\to 0,
\qquad \text{as }\eps\to 0.
\end{align*}
To this end, we exploit the fact that every $u\in\mcU^\eps$ can be approximated arbitrarily well by a point process, in the spirit of Section~4.1.2 in \cite{Fuhrman15} and Section~4.2 in \cite{Fuhrman2020}. However, the game framework considered here necessitates a substantially different approach from those employed in the aforementioned works. The main difficulty stems from the fact that the optimal stopping times $\tau_n$ depend on $n$. To overcome this issue, we introduce a discretization of the set of stopping times $\mcT_t^\mcR$ appearing in the game representation of $Y^n$. More specifically, we restrict stopping times to take values in $\bbT^{t,\eps}$ and simultaneously reduce the available information by considering stopping times with respect to a smaller filtration.

Our approach is specifically tailored to the game setting and therefore differs significantly from the methods of \cite{Fuhrman15,Fuhrman2020}. In this regard, it is closer in spirit to \cite{imp-stop-game}, which studied a game between an impulse controller and a stopper. Nevertheless, the fundamental differences between impulse control and classical control require a substantially different analysis, which we develop below.

\subsection{A discretization of the randomized game\label{subsec:disc-mcR}}
To prove that the inequality in \eqref{ekv:crand-dom} holds, we fix $(t,\bx)\in[0,T]\times\bfC^d$. For each $\eps >0$, we let
\begin{align*}
  \bfU^\eps:=\{\bold u:[t,T]\to U: \exists (\alpha_i)_{i=1}^{n^{t,\eps^2}_\bbT-1}\subset \bar U^\eps,\,\bold u(s)=\sum_{i=1}^{n^{t,\eps^2}_\bbT-2}\alpha_i\ett_{[t^{t,\eps^2}_i,t^{t,\eps^2}_{i+1})}(s) + \alpha_{n^{t,\eps^2}_\bbT-1}\ett_{[t^{t,\eps^2}_{n^{t,\eps^2}_\bbT-1},t^{t,\eps^2}_{n^{t,\eps^2}_\bbT}]}(s)\}
\end{align*}
and let $\bbF^{\bfU^\eps}:=(\mcF^{\bfU^\eps}_s)_{s\in [t,T]}$ be the filtration generated by the coordinate map, $\mcF_s^{\bfU^\eps}:=\sigma(\bold u\mapsto \bold u(r):t\leq r\leq s)$, on $\bfU^\eps$. We introduce the map $\Xi^{\mcR,\eps}:\{ \bold u:[t,T]\to U\} \to \bfU^\eps$ defined as
\begin{align*}
  \Xi^{\mcR,\eps}[\bold u](s)=\sum_{i=1}^{n^{t,\eps^2}_\bbT-2}\ett_{[t^{t,\eps^2}_i,t^{t,\eps^2}_{i+1})}(s)\sum_{j=1}^{n^\eps_U} b^\eps_j\ett_{U^\eps_j}(\bold u(t^{t,\eps^2}_{i})) + \ett_{[t^{t,\eps^2}_{n^{t,\eps^2}_\bbT-1},T]}(s)\sum_{j=1}^{n^\eps_U} b^\eps_j\ett_{U^\eps_j}(\bold u(t^{t,\eps^2}_{n^{t,\eps^2}_\bbT-1})),\quad\forall s\in [t,T],
\end{align*}
and let $I^{t,\eps}:=\Xi^{\mcR,\eps}[I^t]$ so that, for small $\eps>0$, the process $I^{t,\eps}$ provides a time-discretized approximation of $I^t$ with mesh $\Delta^{t,\eps^2}$, which is finer than the discretization $\Delta^{t,\eps}$ used in the primal problem, while the discretization of $U$ remains unchanged. We then let
\begin{align*}
  X^{t,\bx,\eps}_s=\bx(s\wedge t)+\int_t^{s\vee t} a(r,X^{t,\bx,\eps},I^{t,\eps}_r)dr+\int_t^{s\vee t}\sigma(r,X^{t,\bx,\eps},I^{t,\eps}_r)dW_r,\quad\forall s\in [0,T]
\end{align*}
and define the corresponding cost/reward functional
\begin{align*}
  J^{\mcR,\eps}(t,\bx;\nu,\tau)&:=\E^\nu\Big[\psi(\tau,X^{t,\bx,\eps}) + \int_t^{\tau}f(s,X^{t,\bx,\eps},I^{t,\eps}_s)ds\,\Big|\,\mcF_t\Big].
\end{align*}

We restrict the admissible densities in the randomized setting and introduce
\begin{align*}
  \mcV^{n,\eps}:=\{\nu\in\mcV^n:\nu_s(e)\leq\eps,\,\forall (s,e)\in \cup_{i=1}^{n^{t,\eps}_\bbT-1}(t^\eps_i+\Delta^{t,\eps^2},t^\eps_{i+1})\times U\}.
\end{align*}
Note that for small $\eps>0$, any $\nu\in\mcV^{n,\eps}$ induces a probability measure $\Prob^\nu$ under which the randomized control changes value only with small probability ($<\eps T$) on the interval $\cup_{i=1}^{n^\eps_{\bbT}-1}(t^\eps_i+\Delta^{t,\eps^2},t^\eps_{i+1})$. In this sense, the resulting randomized control $I^{t,\eps}$, resembles a control with time-discretization using a step size $\Delta^{t,\eps}$ under $\Prob^\nu$.

Finally, we define the corresponding discretized, truncated value function as
\begin{align}\label{ekv:stop-proc-disc}
  v^{\mcR,n,\eps}(t,\bx)= \esssup_{\tau\in\mcT^{\mcR}_t}\essinf_{\nu\in\mcV^{n,\eps}}J^{\mcR,\eps}(t,\bx;\nu,\Xi^{t,\eps}_\bbT(\tau)).
\end{align}
Note that the stopping times are projected onto $\bbT^{t,\eps}$ via $\Xi^{t,\eps}_\bbT$, ensuring consistency with the discretized framework.

We will make frequent use of the following equivalent to Proposition~\ref{prop:Xu-stab}:

\begin{lem}\label{lem:dual-X-stab}
For each $p\geq 1$, there is a constant $C_p>0$, such that
\begin{align}\label{ekv:moment-X-dual}
  \E^\nu\Big[\|X^{t,\bx}\|^p_T+\|X^{t,\bx,\eps}\|^p_T\,\Big|\, \mcF^\mcR_t\Big]\leq C_p(1+\|\bx\|^p_t)
\end{align}
for all $(t,\bx)\in [0,T]\times\bfC^d$ and $\nu\in\mcV$. Moreover, letting $\mcV^{\infty,\eps}:=\cup_{n=1}^\infty\mcV^{n,\eps}$, we have
\begin{align*}
  \esssup_{\nu\in\mcV^{\infty,\eps}}\E^\nu\Big[\int_t^T|I^t_s-I^{t,\eps}_s|^2ds+\|X^{t,\bx}-X^{t,\bx,\eps}\|^2_T\,\Big|\, \mcF^\mcR_t\Big]\to 0,\quad\Prob-\text{a.s. as }\eps\to 0
\end{align*}
for any $(t,\bx)\in [0,T]\times\bfC^d$.
\end{lem}

\begin{proof}
Since the change of measure from $\Prob$ to $\Prob^\nu$ only affects the distribution of $I^t$ (resp. $I^{t,\eps}$), the first property follows by the linear growth of $a$ and $\sigma$.

For the second inequality, we pick $\nu\in\mcV^{\infty,\eps}$ and note that arguing as in the proof of Proposition~\ref{prop:Xu-stab} gives
\begin{align}\nonumber
  \E^\nu\Big[\sup_{s\in [0,T]}|X^{t,\bx}_s-X^{t,\bx,\eps}_s|^2\,\Big|\, \mcF^\mcR_t\Big]
  &\leq C\E^\nu\Big[\int_{t}^{T} (|a(r,X^{t,\bx},I^{t}_r)-a(r,X^{t,\bx},I^{t,\eps}_r)|^2
  \\
  &\quad + |\sigma(r,X^{t,\bx},I^{t}_r)-\sigma(r,X^{t,\bx},I^{t,\eps}_r)|^2)dr\,\Big|\, \mcF^\mcR_t\Big].\label{ekv:after-Gronwall}
\end{align}
On the other hand, for any $\nu\in\mcV^{\infty,\eps}$, we have that
\begin{align*}
  \E^{\nu}\Big[\int_t^T|I^t_s-I^{t,\eps}_s|ds\,\Big|\, \mcF^\mcR_t\Big]\leq C(\eps+n^{t,\eps}_\bbT\Delta^{t,\eps^2}),
\end{align*}
where the right-hand side tends to 0 as $\eps\to 0$. Hence, for any sequence $\eps_i$ that decreases to 0 and any sequence $(\nu^i)_{i\in\bbN}$ with $\nu^i\in\mcV^{\infty,\eps_i}$ we have that% the following convergence in probability (under $(\Prob^{\nu^i})_{i\in\bbN}$)
\begin{align*}
  \E^{\nu^i}\Big[\int_t^T\ett_{[|I^t_s-I^{t,\eps_i}_s|>\delta]}ds\,\Big|\, \mcF^\mcR_t\Big]\to 0,\quad \text{as }i\to\infty
\end{align*}
for any $\delta>0$. Letting $(\nu^i)_{i\in\bbN}$ be a maximizing sequence for the right-hand side of \eqref{ekv:after-Gronwall} and using local uniform continuity of $a$ and $\sigma$ together with the moment estimate \eqref{ekv:moment-X-dual}, the desired result follows.
\end{proof}

The following lemma connects the value for the discretized dual game to that of the original dual game
\begin{lem}\label{lem:vR-approx}
$\sup_{n\in\bbN}\E\big[(v^{\mcR,n}(t,\bx)-v^{\mcR,n,\eps}(t,\bx))^+\big]\to 0$ as $\eps\to 0$.
\end{lem}

\begin{proof}[Proof.]
We have,
\begin{align*}
  v^{\mcR,n}(t,\bx)-v^{\mcR,n,\eps}(t,\bx)& = \essinf_{\nu\in\mcV^{n}}J^{\mcR}(t,\bx;\nu,\tau_n) - \esssup_{\tau\in\mcT^{\mcR}_t}\essinf_{\nu\in\mcV^{n,\eps}}J^{\mcR,\eps}(t,\bx;\nu,\Xi^{t,\eps}_\bbT(\tau))
  \\
  & \leq \esssup_{\nu\in\mcV^{n,\eps}}\{J^{\mcR}(t,\bx;\nu,\tau_n) - J^{\mcR,\eps}(t,\bx;\nu,\Xi^{t,\eps}_\bbT(\tau_n))\}
  \\
  &\leq\esssup_{\nu\in\mcV^{\infty,\eps}}\E^\nu\Big[\psi(\tau_n,X^{t,\bx}) -\psi(\Xi^{t,\eps}_\bbT(\tau_n),X^{t,\bx,\eps})
  \\
  &\quad + \int_t^{\tau_n}(f(s,X^{t,\bx},I^t_s)-f(s,X^{t,\bx,\eps},I^{t,\eps}_s))ds - \int_{\tau_n}^{\Xi^{t,\eps}_\bbT(\tau_n)}f(s,X^{t,\bx,\eps},I^{t,\eps}_s)ds\,\Big|\,\mcF^\mcR_t\Big].
\end{align*}
Hence, for any $\delta>0$, there is by \eqref{ekv:moment-X-dual} and polynomial growth a $K>0$ (which is independent of $\eps$ and $n$) such that
\begin{align*}
  &v^{\mcR,n}(t,\bx)-v^{\mcR,n,\eps}(t,\bx)
  \\
  &\leq\esssup_{\nu\in\mcV^{\infty,\eps}}\E^\nu\Big[\ett_{[\|X^{t,\bx}\|_{T}\vee \|X^{t,\bx,\eps}\|_{T}\leq K]}\Big(\psi(\tau_n,X^{t,\bx}) -\psi(\Xi^{t,\eps}_\bbT(\tau_n),X^{t,\bx,\eps})
  \\
  &\quad + \int_t^{\tau_n}(f(s,X^{t,\bx},I^t_s)-f(s,X^{t,\bx},I^{t,\eps}_s))ds - \int_{\tau_n}^{\Xi^{t,\eps}_\bbT(\tau_n)}f(s,X^{t,\bx,\eps},I^{t,\eps}_s)ds\Big)\,\Big|\,\mcF^\mcR_t\Big]+\delta
  \\
  &\leq \esssup_{\nu\in\mcV^{\infty,\eps}}\E^\nu\Big[\varpi_K(\mathbf d_\Lambda[(\tau_n,X^{t,\bx}),(\Xi^{t,\eps}_\bbT(\tau_n),X^{t,\bx,\eps})]\wedge K)+
  \\
  &\quad + \int_t^{\tau_n}\varpi_K((\|X^{t,\bx}-X^{t,\bx,\eps}\|_s+|I^t_s-I^{t,\eps}_s|)\wedge K)ds\,\Big|\,\mcF^\mcR_t\Big]+C\eps+\delta.
\end{align*}
Finally, the result follows by Lemma~\ref{lem:dual-X-stab}.
\end{proof}

\bigskip

In the definition of the discretized value function in \eqref{ekv:stop-proc-disc}, the information available to the stopper can be restricted in a suitable way without affecting the value. For any $\bbF^{\mcR}$-progressively measurable process $(R_s:0\leq s\leq T)$, we introduce the sub-filtration $\bbF^{W,R}:=(\mcF^{W,R}_s)_{0\leq s\leq T}$ generated by $W$ and $R$ and augmented with all $\Prob$-null sets and note that when $R=I^{t,\eps}$ both $I^{t,\eps}$ and $X^{t,\bx,\eps}$ are $\bbF^{W,I^{t,\eps}}$-adapted processes. We then let $\mcT^{\mcR,\eps}_t$ be the set of $\bbF^{W,I^{t,\eps}}$-stopping times valued in $\bbT^{t,\eps}$ and have the following lemma:

\begin{lem}\label{lem:stop-proc-disc}
For $\eps>0$, there is a non-increasing sequence $(\tau^{\eps}_n)_{n\in\bbN}$ in $\mcT^{\mcR,\eps}_t$ such that
\begin{align*}
  v^{\mcR,n,\eps}(t,\bx)= \essinf_{\nu\in\mcV^{n,\eps}}J^{\mcR,\eps}(t,\bx;\nu,\tau^{\eps}_n)
\end{align*}
for each $n\in\bbN$.
\end{lem}

\begin{proof}
To show the existence of an optimal stopping time we use dynamic programming and introduce the processes
\begin{align*}
\tilde Y^{n,\eps}_s&:=\esssup_{\tau\in\mcT^\mcR_s}\essinf_{\nu\in\mcV^{n,\eps}}J^{\mcR,\eps}(s,X^{t,\bx,\eps};\nu,\Xi_\bbT^{t,\eps}(\tau)),\quad \forall s\in [t,T]
\end{align*}
from which we extract the stopping times
\begin{align*}
\tau^\eps_n&:=\inf\{s\in \bbT^{t,\eps}:\tilde Y^{n,\eps}_s=\psi(s,X^{t,\bx,\eps})\}.
\end{align*}
Since $\tilde Y^{n,\eps}_s$ is non-increasing in $n$, the sequence of stopping times $(\tau_n^{\eps})_{n\in\bbN}\in\mcT^{\mcR,\eps}_{t}$ is non-increasing in $n$. Now, in the left-hand side of \eqref{ekv:stop-proc-disc}, stopping outside of the set $\bbT^{t,\eps}$ is suboptimal from the point of view of the stopper. On $\bbT^{t,\eps}$, we thus have
\begin{align*}
\tilde Y^{n,\eps}_{t^\eps_i}&=\esssup_{\tau\in\mcT^\mcR_{t^\eps_i}}\essinf_{\nu\in\mcV^{n,\eps}} \{\ett_{[\tau={t^\eps_i}]}J^{\mcR,\eps}(t^\eps_i,X^{t,\bx,\eps};\nu,{t^\eps_i}) + \ett_{[\tau>{t^\eps_i}]}J^{\mcR,\eps}(t^\eps_i,X^{t,\bx,\eps};\nu,\Xi^{t,\eps}_\bbT(\tau))\}
\\
&=\psi({t^\eps_i},X^{t,\bx,\eps})\vee \esssup_{\tau\in\mcT^\mcR_{t^\eps_{i+1}}}\essinf_{\nu\in\mcV^{n,\eps}}\E^\nu\Big[\int_{t^\eps_{i}}^{t^\eps_{i+1}} f(r,X^{t,\bx,\eps},I^{t,\eps}_r)dr +J^{\mcR,\eps}(t^\eps_{i+1},X^{t,\bx,\eps};\nu,\Xi^{t,\eps}_\bbT(\tau))\,\Big|\,\mcF^\mcR_{t^\eps_i}\Big].
\end{align*}
By a regular BSDE argument, the non-linear expectation $\essinf_{\nu\in\mcV^{n,\eps}}\E^\nu$ satisfies a tower property and we get that for arbitrary $\tau\in\mcT^\mcR_{t^\eps_i}$,
\begin{align*}
&\essinf_{\nu\in\mcV^{n,\eps}}\E^\nu\Big[\int_{t^\eps_{i}}^{t^\eps_{i+1}} f(r,X^{t,\bx,\eps},I^{t,\eps}_r)dr +J^{\mcR,\eps}(t^\eps_{i+1},X^{t,\bx,\eps};\nu,\Xi^{t,\eps}_\bbT(\tau))\,\Big|\,\mcF^\mcR_{t^\eps_i}\Big]
\\
&=\essinf_{\nu\in\mcV^{n,\eps}}\E^\nu\Big[\int_{t^\eps_{i}}^{t^\eps_{i+1}} f(r,X^{t,\bx,\eps},I^{t,\eps}_r)dr +\essinf_{\nu'\in\mcV^{n,\eps}}J^{\mcR,\eps}(t^\eps_{i+1},X^{t,\bx,\eps};\nu',\Xi^{t,\eps}_\bbT(\tau))\,\Big|\,\mcF^\mcR_{t^\eps_i}\Big].
\end{align*}
This leads us to the conclusion that $\tilde Y^{n,\eps}$ satisfies the weak dynamic programming principle,
\begin{align}\label{ekv:weakDP}
\tilde Y^{n,\eps}_{t^\eps_i}&\leq \psi({t^\eps_i},X^{t,\bx,\eps})\vee \essinf_{\nu\in\mcV^{n,\eps}}\E^\nu\Big[\int_{t^\eps_{i}}^{t^\eps_{i+1}} f(r,X^{t,\bx,\eps},I^{t,\eps}_r)dr+\tilde Y^{n,\eps}_{t^\eps_{i+1}}\,\Big|\,\mcF^\mcR_{t^\eps_i}\Big].
\end{align}
Letting $i=1$ gives $t^\eps_i=t$. By iteration and again using the tower property, we find that the right-hand side of \eqref{ekv:weakDP} is bounded from above by $\essinf_{\nu\in\mcV^{n,\eps}} J^{\mcR,\eps}(t,\bx;\nu,\tau^\eps_n)$, yielding the conclusion that
\begin{align*}
  v^{\mcR,n,\eps}(t,\bx)=\tilde Y^{n,\eps}_{t}=\essinf_{\nu\in\mcV^{n,\eps}}J^{\mcR,\eps}(t,\bx;\nu,\tau^\eps_n).
\end{align*}
Moreover, standard arguments give that $\tilde Y^{n,\eps}_s$ is $\mcF^{W,I^{t,\eps}}_s$-measurable whenever $s\in \bbT^{t,\eps}$, implying that $\tau^\eps_n$ is an $\bbF^{W,I^{t,\eps}}$-stopping time and thus belongs to $\mcT^{\mcR,\eps}_{t}$ for each $n\in\bbN$.
\end{proof}

\bigskip

\subsection{An auxiliary probability space\label{subsec:proof-of-prop}}

Inspired by Section 4.3 of \cite{Fuhrman2020}, we introduce an auxiliary probability space $( \Omega',\mcF',\Prob')$ on which lives real-valued random variables $(U^m_j,S^m_j)_{m,j\in\bbN}$ and random measures $(\pi^l)_{l\in\bbN}$ such that
\begin{enumerate}
  \item the $U^m_j$ are all uniformly distributed on $(0,1)$,
  \item the probability distribution of $S^m_j$ admits a density $f^m_j$ with respect to the Lebesgue measure, that has support on the interval $((1-2^{1-j})/m,(1-2^{-j})/m)$, so that $0<S^m_1<S^m_2<\cdots<1/m$ for every $m\in\bbN$,
  \item every $\pi^l$ is a Poisson random measure on $(0,\infty)\times U$, with compensator $l^{-1}\lambda(da)dt$, with respect to its natural filtration;
  \item the random elements $U^m_j,S^{m'}_{j'},\pi^l$ are all independent.
\end{enumerate}
Now, we define $\hat\Omega:=\Omega\times\Omega'$, let $\hat \mcF$ be the $\Prob\otimes\Prob'$ completion of $\mcF\otimes\mcF'$ and let $\hat\Prob$ denote the extension of $\Prob\otimes\Prob'$ to $\hat\mcF$. Further, we let $\hat W,\hat\mu,\hat U^m_j,\hat S^{m'}_{j'}$ and $\hat \pi^l$ denote the canonical extensions of $W,\mu,U^m_j,S^{m'}_{j'}$ and $\pi^l$ to $\hat\Omega$. For $u\in\check\mcU_t$ and $\tau\in\check\mcT_t$ (which are extensions of $\mcU_t$ and $\mcT_t$ to $\hat\Omega$, that are more carefully defined below), we define
\begin{align*}
  \hat J(t,\bx;u,\tau)=\hat\E\Big[\psi(\tau,\hat X^{t,\bx;u})+\int_t^{\tau} f(r,\hat X^{t,\bx;u},u_r)dr\,\Big|\,\hat\mcF_t\Big],
\end{align*}
where $\hat\E$ is expectation with respect to $\hat\Prob$, the filtration $\hat\bbF:=(\hat \mcF_t)_{{t\geq 0}}$ is the $\hat\Prob$-augmented natural filtration on $(\hat\Omega,\hat\mcF)$ generated by $\hat W$ and $\hat X^{t,\bx;u}$ solves
\begin{align*}
  \hat X^{t,\bx;u}_s=\bx(s\wedge t)+\int_t^{t\vee s} a(r,\hat X^{t,\bx;u},u_r)dr+\int_t^{t\vee s}\sigma(r,\hat X^{t,\bx;u},u_r)d\hat W_r,\quad\forall s\in [0,T].
\end{align*}
When extending the basic notations to the space $(\hat\Omega,\hat\mcF,\hat\Prob)$, we introduce two versions of most objects depending on whether they utilize the information in the $\sigma$-algebra $\mcF'$ or not. We denote objects that incorporate the information in $\mcF'$ with a check symbol, while objects that do not use this information are denoted with a hat symbol. Specifically, we make the following definitions:
\begin{itemize}
\item We let $\hat\bbF$ (resp.~$\check\bbF$) be the $\hat\Prob$-completion of the filtration $(\mcF_s \times \Omega')_{s\geq 0}$ (resp.~$(\mcF_s\otimes \mcF')_{s\geq 0}$).
\item For any $\Prog(\check\bbF)$-measurable process $(\check R_s:0\leq s\leq T)$, we introduce the sub-filtration $\check\bbF^{\hat W,\check R}:=(\check\mcF^{\hat W,\check R}_s)_{0\leq s\leq T}$ of $\check\bbF$ generated by $\hat W$ and $\check R$ and augmented with all $\hat\Prob$-null sets.
\item We let $\hat\mcT_t$ be the set of all $\hat\bbF$-stopping times $\tau$ with $\tau\in [t,T]$, $\hat\Prob$-a.s.
\item We let $\hat\mcU_t$ (resp. $\check\mcU_t$) be the set of all $\Prog(\hat\bbF)$-measurable (resp.~$\Prog(\check\bbF)$-measurable) processes $(u_s:t\leq s\leq T)$ valued in $U$.
\item For each $\eps>0$, $\hat\mcU^{\eps}_t$ is defined as
    \begin{align*}
        \hat\mcU^{\eps}_t:=\Big\{u\in\hat\mcU_t: \exists (\beta_i)_{i=1}^{n_{\bbT}^\eps-1},\,\beta_i\in m\hat\mcF_{t^\eps_i},\, u_s= \sum_{i=1}^{n_{\bbT}^\eps-2}\beta_{i}\ett_{[t^\eps_i,t^\eps_{i+1})}(s) + \beta_{n_{\bbT}^\eps-1}\ett_{[t^\eps_{n_{\bbT}^\eps-1},T]}(s)\Big\}.
\end{align*}
\item For each $\check u\in\check \mcU_t$, we denote by $\check\mcT^{\check u}_t$ the set of $\check\bbF^{\hat W,\check u}$-stopping times, $\tau$, such that $\tau\geq t$, $\hat\Prob$-a.s.~and we let $\check\mcT^{\check u,\eps}_t$ be the restriction to stopping times that are stopping times with respect the smaller filtration $\check\bbF^{\hat W,\Xi^{\mcR,\eps}[\check u]}$ and take values in $\bbT^{t,\eps}$.
\end{itemize}

Following the above procedure, we define $\hat{\underline v}^{\eps}(t,\bx)$ as the canonical extension of ${\underline v}^{\eps}(t,\bx)$ to $\hat\Omega$. We fix $\varrho>0$ and note that for any $\eps>0$, there is a $\hat u^{\varrho,\eps}\in\hat\mcU^\eps$ such that
\begin{align}\label{ekv:u-varrho-eps-def}
  \hat{\underline v}^{\eps}(t,\bx)\geq \hat J(t,\bx;\hat u^{\varrho,\eps},\tau)-\varrho,
\end{align}
for any $\tau\in\hat\mcT_t$.

The idea is to use the sequences $\hat U^m_j$ and $\hat S^{m'}_{j'}$ to ``randomize'' $\hat u^{\varrho,\eps}$ and then add the jumps in $\hat \pi^l$ to obtain a new point process $\check u\in\check\mcU$ such that the $\hat\Prob$-compensator of the corresponding random measure, \ie the unique random measure $\check\mu$ (that separates points in time) such that $\check u_s=\beta_0+\int_0^s\int_U(e-\check u_r)\check\mu(dr,de)$), has a density $\check\nu$ with respect to $\lambda(da)dt$ which is bounded from below by a positive constant and such that $\check u$ is sufficiently ``close'' to $\hat u$.

For each $m\geq 1$, define the kernel $q^m:(b,da)\mapsto \frac{1}{\lambda({\bold B}(b,1/m))}\ett_{{\bold B}(b,1/m)}(a)\lambda(da)$ (where ${\bold B}(b,1/m)$ is the closed ball of radius $1/m$, centered at $b$) as in the proof of Lemma 4.4 of \cite{Fuhrman2020}. Using the sequence $(q^m)_{m\in\bbN}$, we define the sequence of controls $(\check u^{\varrho,\eps,m})_{m\in\bbN}$ as the piecewise constant processes
\begin{align*}
  \check u^{\varrho,\eps,m}:=\beta_0\ett_{[t,\check\eta^{\varrho,\eps,m}_1)}+\sum_{j\geq 1}\check\theta^{\varrho,\eps,m}_j\ett_{[\check\eta^{\varrho,\eps,m}_j,\check\eta^{\varrho,\eps,m,l}_{j+1})},
\end{align*}
where
\begin{align*}
 \begin{cases}
  \check \eta^{\varrho,\eps,m}_1:=t+\hat S^m_1
  \\
  \check \eta^{\varrho,\eps,m}_j:=(\check \eta^{\varrho,\eps,m}_{j-1}\vee t^\eps_j)+\hat S^m_j,\quad j>1,
  \\
  \check\theta^{\varrho,\eps,m}_j:=q^m(\hat u^{\varrho,\eps}_{t^\eps_j},\hat U^m_j),\quad\forall j\in\bbN.
 \end{cases}
\end{align*}

An important feature of the above definition is that $\check u^{\varrho,\eps,m}_s$ lies within a $1/m$-neighborhood of $\hat u^{\varrho,\eps}_s$ whenever $s\in \cup_{j=1}^{n^{t,\eps}_\bbT-1}[\check\eta^{\varrho,\eps,m}_j,t^\eps_{j+1})$. Since $\check\eta^{\varrho,\eps,m}_j\searrow t^\eps_j$, $\hat\Prob$-a.s.,~as $m\to\infty$, we conclude that
\begin{align*}
  \check\rho_t(\check u^{\varrho,\eps,m},\hat u^{\varrho,\eps})\to 0,\quad\text{as }m\to\infty,
\end{align*}
where $\check\rho_t:\check\mcU_t\times\check\mcU_t\to\R_+$ is the pseudo-metric on $\check\mcU_t$ defined as
\begin{align*}
  \check\rho_t(u,\tilde u):=\hat\E\Big[\int_{t}^T|u_s-\tilde u_s| ds\Big].
\end{align*}
Consequently, possibly after passing to a subsequence,
\begin{align*}
  \hat J(t,\bx;\check u^{\varrho,\eps,m},\tau)\to\hat J(t,\bx;\hat u^{\varrho,\eps},\tau),\quad\hat\Prob-\text{a.s.~as }m\to\infty
\end{align*}
for any $\tau\in\check\mcT_t$. %Furthermore, whenever $m$ is sufficiently large we find that $\Xi_{\check\mcU}^\eps[\check u^{\varrho,\eps,m}]=\hat u^{\varrho,\eps}$ pointwisely, $\hat\Prob$-a.s.

The above control induces a random measure on $[t,T]\times U$ defined as,
\begin{align*}
  \mu^{{\varrho,\eps,m}}:=\sum_{j=1}^{n^{t,\eps}_\bbT-1}\delta_{(\check \eta^{\varrho,\eps,m}_j,\check \theta^{\varrho,\eps,m}_j)}.
\end{align*}
According to Lemma A.11 in~\cite{Fuhrman15}, the random measure $\mu^{{\varrho,\eps,m}}$ has a $\hat\Prob$-compensator with respect to $\check\bbF^{\hat W,\check u^{\varrho,\eps,m}}$ given by the explicit formula
\begin{align*}
  \sum_{j=1}^{n^{t,\eps}_\bbT-1}\ett_{(\check \eta^{\varrho,\eps,m}_{j-1}\vee t^\eps_j ,\check \eta^{\varrho,\eps,m}_j]}(t)q^m(\hat u^{\varrho,\eps}_{t^\eps_j},da)\frac{f^m_j(s-(\check \eta^{\varrho,\eps,m}_{j-1}\vee t^\eps_j))}{1-F^m_j(s-(\check \eta^{\varrho,\eps,m}_{j-1}\vee t^\eps_j))}ds,
\end{align*}
with $F^m_j(s):=\int_{-\infty}^s f^m_j(r)dr$. For each $m\in\bbN$, this compensator is clearly $\Pred(\check\bbF^{\hat W,\check u^{\varrho,\eps,m}}) \otimes\mcB(U)$-measurable. However, the density equals zero on $\cup_{j=1}^{n^{t,\eps}_\bbT-1}[\check\eta^{\varrho,\eps,m}_j,t^\eps_{j+1})$ and is, therefore, not bounded away from zero, which turns out to be a necessity in subsequent developments. To remedy this we add the jumps in $\hat\pi^l$ to obtain the random measure $\check\mu^{\varrho,\eps,m,l}:=\mu^{\varrho,\eps,m}+\hat\pi^l$ that corresponds to the randomized control
\begin{align*}
  \check I^{m,l}_s=\beta_0+\int_t^s\!\!\int_U(e-\check I^{m,l}_{r-})\check\mu^{\varrho,\eps,m,l}(dr,de),\quad\forall s\in [t,T].\\
\end{align*}
With this definition, $\check\mu^{\varrho,\eps,m,l}$ has $\hat\Prob$-compensator with respect to the filtration $\check\bbF^{\hat W,\check I^{m,l}}$ that is absolutely continuous with respect to $\lambda$ and takes the form
\begin{align*}
\check\nu^{m,l}_s(\hat\omega,e)\lambda(de)ds
\end{align*}
where $\check\nu^{m,l}$ is $\Pred(\check\bbF^{\hat W,\check I^{m,l}}) \otimes\mcB(U)$-measurable and bounded from below by $1/l$. On the other hand,
\begin{align*}
  \check\rho_t(\check I^{m,l},\hat u^{\varrho,\eps})\to 0,\quad\text{as }m,l\to\infty.
\end{align*}
Again, we thus find by possible going to a subsequence, that
\begin{align*}
  \hat J(t,\bx;\check I^{m,l},\tau)\to\hat J(t,\bx;\hat u^{\varrho,\eps},\tau),\quad\hat\Prob-\text{a.s.~as }m,l\to\infty
\end{align*}
for any $\tau\in\check\mcT_t$. %Furthermore, whenever $m$ is sufficiently large we find that $\hat\Xi_{\check\mcU}^\eps[\check u^{\varrho,\eps,m}]=\hat u^{\varrho,\eps}$ pointwisely on the set $\check A^l:=\{\omega\in\Omega:\pi^l([t,T],U)=0\}$.

Finally, for $m$ sufficiently large we get that $\check \eta^{\varrho,\eps,m}_j\in (t^\eps_j,t^\eps_j+\Delta^{t,\eps^2})$ in which case we find that $\check\nu^{m,l}_s=1/l$ on $\cup_{j=1}^{n^{t,\eps}_\bbT-1}[t^\eps_j+\Delta^{t,\eps^2},t^\eps_{j+1}]$. We summarize these findings in the following lemma:

\begin{lem}\label{lem:will-conv}
Whenever $\eps>0$ is sufficiently small, there is a piecewise constant process
\begin{align*}
  \check u^{\varrho,\eps}:=\beta_0\ett_{[t,\eta_1)}+\sum_{j\geq 1}\theta_j\ett_{[\eta_j,\eta_{j+1})}\in\check\mcU,
\end{align*}
with $(\eta_j)_{j\in\bbN}$ strictly increasing and a set $A^{\varrho,\eps}\in\check\mcF$ with $\hat P[A^{\varrho,\eps}]\geq 1-\eps$ such that:
\begin{enumerate}[a)]
  \item $\hat\E\big[(\hat J(t,\bx;\Xi^{\mcR,\eps}[\check u^{\varrho,\eps}],\tau)-\hat J(t,\bx;\hat u^{\varrho,\eps},\tau))^+\big]\leq\varrho$ for each $\tau\in\check\mcT_t$.
  \item\label{enum:disc-equiv} Whenever, $\omega\in A^{\varrho,\eps}$, we have
  \begin{align*}
    (\Xi^{\mcR,\eps}[\check u^{\varrho,\eps}])_s=\beta_0\ett_{[0,t+\Delta^{t,\eps^2})}(s)+\hat u^{\varrho,\eps}_{s-\Delta^{t,\eps^2}}\ett_{[t+\Delta^{t,\eps^2},T]}(s),\quad\forall s\in [0,T].
  \end{align*}
\end{enumerate}
Moreover, the random measure on $[t,T]\times U$ corresponding to $\check u^{\varrho,\eps}$, \ie $\check\mu^{{\varrho,\eps}}:=\sum_{j\geq 1} \delta_{(\eta_j,\theta_j)}$ has a $\hat\Prob$-compensator with respect to the filtration $\check\bbF^{\hat W,\check u^{\varrho,\eps}}$ that is absolutely continuous with respect to $\lambda$ and takes the form
\begin{align*}
\check\nu^{\varrho,\eps}_s(\hat\omega,e)\lambda(de)ds
\end{align*}
where $\check\nu^{\varrho,\eps}$ is $\Pred(\check\bbF^{\hat W,\check u^{\varrho,\eps}}) \otimes\mcB(U)$-measurable, bounded away from zero and less than $\eps$ on $\cup_{j=1}^{n^{t,\eps}_\bbT-1}[t^\eps_j+\Delta^{t,\eps^2},t^\eps_{j+1}]$.
\end{lem}

\begin{proof}
The result follows immediately by choosing $m$ and $l$ sufficiently large in the preceding constructions.
\end{proof}

To limit notation we drop the superscripts and set $\check u:=\check u^{\varrho,\eps}$ and $\check\nu=\check\nu^{\varrho,\eps}$. To establish a correspondence between the original game and its randomized version, we combine $\pi_1$ and $\mu^{\varrho,\eps}$ to obtain the random measure $\check\mu:=\pi_1(\cdot\cup [0,t],\cdot)+\mu^{\varrho,\eps}(\cdot\cup (t,T],\cdot)$. By Lemma~\ref{lem:will-conv}, $\check\mu$ has a $\hat\Prob$-compensator with respect to the filtration $\hat\bbF^{\mcR}$, the latter being the natural filtration on $(\hat\Omega,\check\mcF)$ generated by $\hat W$ and $\check u$, completed with all $\hat\Prob$-null sets. Moreover, this compensator has a density $\ett_{[0,t]}+\ett_{(t,T]}\check\nu$ with respect to $\lambda$. We abuse notation and use $\check \nu$ to denote this density, the infimum of which is strictly positive whereas the supremum may be unbounded.

The above construction allows us to define an auxiliary randomized version of the game. We denote by $\check\bbF^{\mcR}:=\check\bbF^{\hat W,\check u}$ the filtration generated by $\hat W$ and $\check u$, augmented with all $\hat\Prob$-null sets. Letting $(\check\sigma_j,\check\zeta_j)_{j\geq 1}$ be the marks of $\check\mu$ we find, since $\check\nu$ is bounded from below, that
\begin{align*}
\hat M_s:=\exp\Big(\int_{0}^s\!\int_U(1-(\check\nu_r(a))^{-1})\lambda(da)dr\Big)\prod_{\check\sigma_j\leq s}(\check\nu_{\check\sigma_j}(\check\zeta_j))^{-1}
\end{align*}
is a strictly positive martingale with respect to the filtration $\check\bbF^{\mcR}$ under $\hat\Prob$. Furthermore, as $\check\nu_s\equiv 1$ for all $s\in [0,t]$, we have $\hat M\equiv 1$ on $[0,t]$. We define the equivalent probability measure $\check\Prob$ on $(\hat\Omega,\hat\mcF)$ as $d\check\Prob=\hat M_Td\hat\Prob$. By the Girsanov theorem, $\check\mu$ has $\check\Prob$-compensator $\lambda(da)ds$ with respect to the filtration $\check\bbF^\mcR$. Moreover, despite the fact that $\check\nu$ is generally not bounded we still have a Dol{\'e}ans-Dade exponential
\begin{align*}
\hat\kappa^{\check\nu}_s:=\exp\Big(\int_{t}^s\!\int_U(1-\check\nu_r(e))\lambda(de)dr\Big)\prod_{t<\check\sigma_j\leq s}\check\nu_{\check\sigma_j}(\check\zeta_j)
\end{align*}
for which $\check \E[\hat\kappa^{\check\nu}_T]=\hat \E[\hat M_T\hat\kappa^{\check\nu}_T]=1$, proving that $\hat\kappa^{\check\nu}$ is a $\check\Prob$-martingale. We can thus define a corresponding probability measure, $\check\Prob^{\check\nu}$, on $(\hat\Omega,\hat\mcF)$ as $d\check\Prob^{\check\nu}:=\hat\kappa^{\check\nu}_Td\check\Prob$, and since $\hat M_T\hat\kappa^{\check\nu}_T\equiv 1$, we conclude that $\check\Prob^{\check\nu}=\hat\Prob$ on $(\hat\Omega,\hat\mcF)$. We further extend this definition by letting $d\check\Prob^{\nu}:=\hat\kappa^{\nu}_Td\check\Prob$ whenever $\nu\in\check\mcV$. Here, $\check\mcV$ is the set of all $\check\bbF^\mcR$-predictably measurable bounded maps $\nu=\nu_t(\hat\omega,e):[0,T]\times\hat\Omega\times U\to [0,\infty)$. In particular, with
\begin{align*}
  \begin{cases}
    \check I^t_s=\beta_0+\int_t^s\!\!\int_U(e-\check I^t_{r-})\check \mu(dr,de),\quad\forall s\in [t,T],\\
    \check X^{t,\bx}_s=\bx(s\wedge t)+\int_t^{s\vee t} a(r,\check X^{t,\bx},\check I^t_r)dr+\int_t^{s\vee t}\sigma(r,\check X^{t,\bx},\check I^t_r)d\hat W_r,\quad\forall s\in [0,T]
  \end{cases}
\end{align*}
we get that $\hat J(t,\bx;\check u,\tau)=\check J^{\mcR}(t,\bx;\check\nu,\tau)$, $\check\Prob$-a.s., for each $\tau\in\check\mcT_t^{\mcR}:=\check\mcT_t^{\check u}$, where
\begin{align*}
\check J^{\mcR}(t,\bx;\check\nu,\tau) := \check\E^{\nu}\Big[\psi(\tau,\check X^{t,\bx})+\int_t^{\tau} f(r,\check X^{t,\bx},\check I^t_r)dr\,\Big|\,\hat\mcF_t\Big]
\end{align*}
and $\check\E^{\nu}$ is expectation with respect to $\check\Prob^\nu$.

Since the probability space $(\hat\Omega,\hat\mcF,\check\Prob,\hat W,\check\mu)$ is a setting for our penalized BSDEs \eqref{ekv:rbsde-pen}, there is a unique quadruple $(\check Y^n,\check Z^n,\check V^n,\check K^{-,n})\in \check\mcS^2\times\check\mcH^2(\hat W)\times\check\mcH^2(\check\mu)\times \check\mcA^2$, where $\check\mcS^2$, $\check\mcH^2(\hat W)$, $\check\mcH^2(\check \mu)$ and $\check\mcA^2$ are defined as $\mcS^2$, $\mcH^2(W)$, $\mcH^2(\mu)$ and $\mcA^2$ but on the probability space $(\hat\Omega,\hat\mcF,\check\Prob,\hat W,\check\mu)$, such that
\begin{align}\label{ekv:rbsde-pen-hat}
  \begin{cases}
    \check Y^{t,\bx,n}_s=\psi(T,\check X^{t,\bx})+\int_s^T f(r,\check X^{t,\bx},\check I^t_r)dr-\int_s^T \check Z^{t,\bx,n}_r d\hat W_r-\int_s^T\!\!\!\int_U \check V^{t,\bx,n}_r(e)\check\mu(dr,de)
    \\
    \quad+\check K^{-,{t,\bx,n}}_T-\check K^{-,{t,\bx,n}}_s-n\int_s^T\!\!\!\int_U(\check V^{t,\bx,n}_r(e))^-\lambda(de)dr,\quad \forall s\in [t,T],\\
    \check Y^{t,\bx,n}_s\geq \psi(s,\check X^{t,\bx}),\, \forall s\in [t,T]\mbox{ and } \int_t^T \big(\check Y^{t,\bx,n}_s-\psi(s,\check X^{t,\bx})\big)d\check K^{-,{t,\bx,n}}_s=0.
  \end{cases}
\end{align}
By standard results for reflected BSDEs with jumps, we find that $\check Y^{t,\bx,n}_t=v^{\mcR,n}(t,\bx)$, and by Lemma~\ref{lem-3_4-tilde} we have that
\begin{align}
v^{\mcR,n}(t,\bx)&=\esssup_{\tau\in\check\mcT^{\mcR}_t}\essinf_{\nu\in\check\mcV^n}\check J^{\mcR}(t,\bx;\nu,\tau),\label{ekv:Yn-repr-hat}
\end{align}
employing the obvious notation $\check\mcV^n:=\{\nu\in\check\mcV:\nu\leq n\}$.

\begin{rem}\label{rem:stop-times}
To fully align with the statement of Lemma~\ref{lem-3_4-tilde}, the essential supremum in \eqref{ekv:Yn-repr-hat} should formally be taken over the set of $[t,T]$-valued stopping times with respect to the $\check\Prob$-completion of the filtration generated by $\hat W$ and $\check u$. However, since $\check\Prob$ and $\hat\Prob$ are equivalent measures, this set of stopping times coincides with $\check\mcT^{\mcR}_t$.
\end{rem}

As above, we introduce the discretized version of the randomized control $\check I^{t,\eps}:=\Xi^{\mcR,\eps}[\check I^{t}]$, let
\begin{align*}
  \check X^{t,\bx,\eps}_s=\bx(s\wedge t)+\int_t^{s\vee t} a(r,\check X^{t,\bx,\eps},\check I^{t,\eps}_r)dr+\int_t^{s\vee t}\sigma(r,\check X^{t,\bx,\eps},I^{t,\eps}_r)dW_r,\quad\forall s\in [0,T]
\end{align*}
and define the corresponding cost/reward functional
\begin{align*}
  \check J^{\mcR,\eps}(t,\bx;\nu,\tau)&:=\check\E^\nu\Big[\psi(\tau,\check X^{t,\bx,\eps}) + \int_t^{\tau}f(s,\check X^{t,\bx,\eps},\check I^{t,\eps}_s)ds\,\Big|\,\hat\mcF_t\Big].
\end{align*}
Since $\Xi^{\mcR,\eps}[\check u]=\check I^{t,\eps}:=\Xi^{\mcR,\eps}[\check u]$, pointwisely, $\hat\Prob$-a.s., we have
\begin{align}\label{ekv:primal-dual-eps-comp}
  \hat J(t,\bx;\Xi^{\mcR,\eps}[\check u],\tau)=\check J^{\mcR,\eps}(t,\bx;\check\nu,\tau),\quad \hat\Prob-\text{a.s.},\:\forall \tau\in\check\mcT_t.
\end{align}

Using once more that $(\hat\Omega,\hat\mcF,\check\Prob,\hat W,\check\mu)$ is a setting for our penalized BSDEs \eqref{ekv:rbsde-pen}, Lemma~\ref{lem:vR-approx} and Lemma~\ref{lem:stop-proc-disc} guarantees the existence of a non-increasing sequence $(\check\tau^\eps_n)_{n\in\bbN}$ of stopping times in\footnote{In Lemma~\ref{lem:stop-proc-disc}, these are $\bbT^{t,\eps}$-valued stopping times with respect to the $\check\Prob$-completion of the filtration generated by $\hat W$ and $\Xi^{\mcR,\eps}(\check u)$. However, as noted in Remark~\ref{rem:stop-times}, since $\check\Prob$ and $\hat\Prob$ are equivalent measures, this set of stopping times coincides with $\check\mcT^{\check u,\eps}_t$.} $\check\mcT^{\mcR,\eps}_t:=\check\mcT^{\check u,\eps}_t$ such that\footnote{Employing the obvious notation $\check\mcV^{n,\eps}:=\{\nu\in\check\mcV^n:\nu_s(e)\leq\eps,\,\forall (s,e)\in \cup_{i=1}^{n^{t,\eps}_\bbT-1}(t^\eps_i+\Delta^{t,\eps^2},t^\eps_{i+1})\}$.}
\begin{align}\label{Yn-rel-varrho-2}
 \hat\E\big[(v^{\mcR,n}(t,\bx)- \essinf_{\nu\in\check\mcV^{n,\eps}}\check J^{\mcR,\eps}(t,\bx;\nu,\check\tau^\eps_n))^+\big]\leq \varrho
\end{align}
for each $n\in\bbN$, whenever $\eps>0$ is sufficiently small. On the other hand, by truncating that density $\check \nu$ to get a new density $\check \nu^n:=\check \nu\wedge n$, with $\check \nu^n\in\check\mcV^{n,\eps}$, and we find that
\begin{align*}
  \essinf_{\nu\in\check\mcV^{n,\eps}}\check J^{\mcR,\eps}(t,\bx;\nu,\check\tau^\eps_n)\leq \check J^{\mcR,\eps}(t,\bx;\check\nu^n,\check\tau^\eps_n).
\end{align*}
Consequently,
\begin{align}\label{Yn-rel-varrho-3}
 \hat\E\big[(v^{\mcR,n}(t,\bx)- \check J^{\mcR,\eps}(t,\bx;\check\nu^n,\check\tau^\eps_n))^+\big]\leq \varrho
\end{align}
for all $n\in\bbN$, provided that $\eps>0$ is chosen sufficiently small. Now, as the sequence $(\check\tau^\eps_n)_{n\in\bbN}$ is non-increasing (outside of a $\check\Prob$-null set), there is a stopping time $\check\tau^{*,\eps}\in\check\mcT^{\mcR,\eps}$, such that $\check\tau^\eps_n\searrow \check\tau^{*,\eps}$. Moreover, since the stopping times $(\check\tau^\eps_n)_{n\in\bbN}$ take values in a finite set, the sequence is eventually constant. Hence, the sets $\check A^\eps_n:=\{\omega\in\Omega:\check\tau^{*,\eps}<\check\tau^\eps_n\}$ form a monotone sequence, the limit of which is a $\check\Prob$-null set.

Before we finish the proof of Theorem~\ref{thm:zs-game}, we give two short lemmas.

\medskip

\begin{lem}\label{lem:check-tau-approx}
For any $\eps>0$, we have that $|\check J^{\mcR,\eps}(t,\bx;\check\nu^n,\check\tau^\eps_n)-\check J^{\mcR,\eps}(t,\bx;\check\nu^n,\check\tau^{*,\eps})|\to 0$, $\Prob$-a.s.~as $n\to\infty$.
\end{lem}

\begin{proof}
By polynomial growth, there is a $C>0$ such that for each $\eps>0$ and $n\in\bbN$, we have
\begin{align*}
&\check J^{\mcR,\eps}(t,\bx;\nu,\check \tau^\eps_n)-\check J^{\mcR,\eps}(t,\bx;\nu,\check \tau^{*,\eps})
\\
&= \check \E^\nu\Big[\ett_{\check A^\eps_n}\big(\psi(\check \tau^\eps_n, \check X^{t,\bx,\eps})-\psi(\check \tau^{*,\eps},\check X^{t,\bx,\eps})+\int_{\check \tau^{*,\eps}}^{\check \tau^\eps_n} f(r,\check X^{t,\bx,\eps},\check I^{t,\eps}_r)dr\big)\,\Big|\,\hat\mcF_t\Big]
  \\
  &\leq C\check \E^\nu\big[\ett_{\check A^\eps_n}(1+\|\check X^{t,\bx,\eps}\|_T^q)\big|\hat\mcF_t\big]
  \\
  &\leq C\check \E^\nu\big[\ett_{\check A^\eps_n}\big|\hat\mcF_t\big]^{1/2}(1+\|\bx\|^{q}_t),
\end{align*}
$\check\Prob$-a.s., for all $\nu\in\check\mcV^{n,\eps}$. On the other hand, by dominated convergence we have
\begin{align*}
\int_{0}^T\int_U(1-\check\nu^n_r(e))\lambda(de)dr \to \int_{0}^T\int_U(1-\check\nu_r(e))\lambda(de)dr,
\end{align*}
$\check\Prob$-a.s.~and
\begin{align*}
  \prod_{\hat\sigma_j\leq s}\check\nu^n_{\hat\sigma_j}(\hat\zeta_j)\to \prod_{\hat\sigma_j\leq s}\check\nu_{\hat\sigma_j}(\hat\zeta_j),
\end{align*}
$\check\Prob$-a.s.~as the number of terms in the product is $\check\Prob$-a.s.~finite effectively implying that $\hat\kappa^{\check\nu^n}_T\to \hat\kappa^{\check\nu}_T$, $\check\Prob$-a.s. Since $\check A^\eps_n$ is a monotone sequence of sets the limit of which is a $\check\Prob$-null set we have by dominated convergence that
\begin{align*}
  \lim_{n\to\infty}\check\E^{\check\nu^n}\big[\ett_{\check A^\eps_n}\big|\hat\mcF_{t}\big]=\lim_{n\to\infty}\check\E\big[\hat\kappa^{\check\nu^n}_T\ett_{\check A^\eps_n}\big|\hat\mcF_{t}\big]=0,
\end{align*}
$\check\Prob$-a.s.
\end{proof}

\begin{lem}\label{lem:check-nu-approx}
For any $\eps>0$, we have that $\check J^{\mcR,\eps}(t,\bx;\check\nu^n,\check\tau^{*,\eps})\to\check J^{\mcR,\eps}(t,\bx;\check\nu,\check\tau^{*,\eps})$, $\Prob$-a.s.~as $n\to\infty$.
\end{lem}

\begin{proof}
Letting
\begin{align*}
  \Phi_t(\tau):=\psi(\tau,\check X^{t,\bx,\eps})+\int_t^{\tau} f(r,\check X^{t,\bx,\eps},\check I^{t,\eps}_r)dr,
\end{align*}
we have
\begin{align*}
  |\Phi_t(\tau)|\leq C(1+\|\check X^{t,\bx,\eps}\|^q_T)=:\bar\Phi.
\end{align*}
and get that
\begin{align*}
  |\check J^{\mcR,\eps}(t,\bx;\check\nu^n,\check \tau^{*,\eps})-\check J^{\mcR,\eps}(t,\bx;\check\nu,\check \tau^{*,\eps})|&\leq \esssup_{\tau\in\check\mcT^{\mcR}_t}|\check J^{\mcR,\eps}(t,\bx;\check\nu^n,\tau)-\check J^{\mcR,\eps}(t,\bx;\check\nu,\tau)|
  \\
  &= \esssup_{\tau\in\check\mcT^{\mcR}_t}\check\E\big[|(\hat\kappa^{\check\nu^n}_T-\hat\kappa^{\check\nu}_T)\Phi_t(\tau)|\,\big|\,\hat\mcF_t\big]
  \\
  &\leq \check\E\big[|\hat\kappa^{\check\nu^n}_T-\hat\kappa^{\check\nu}_T|\bar\Phi\,\big|\,\hat\mcF_t\big].
\end{align*}
We will show that the right-hand side tends to 0, $\check\Prob$-a.s., as $n\to\infty$. Letting $E_K:=\{\omega:\|\check X^{t,\bx,\eps}\|_T\leq K\}$, we get
\begin{align*}
  \check\E\big[|\hat\kappa^{\check\nu^n}_T-\hat\kappa^{\check\nu}_T|\bar\Phi\big|\hat\mcF_{t}\big]\leq \check\E\big[\ett_{E_K}|\hat\kappa^{\check\nu^n}_T-\hat\kappa^{\check\nu}_T|\bar\Phi\big|\hat\mcF_{t}\big] + \check\E^{\check\nu^n}\big[\ett_{E_K^c}\bar\Phi\big|\hat\mcF_{t}\big]+ \check\E^{\check\nu}\big[\ett_{E_K^c}\bar\Phi\big|\hat\mcF_{t}\big].
\end{align*}
Concerning the middle term, we have
\begin{align*}
  \check\E^{\check\nu^n}\big[\ett_{E_K^c}\bar\Phi\big|\hat\mcF_{t}\big]&= C\check\E^{\check\nu^n}\big[\ett_{E_K^c}(1+\|\check X^{t,\bx,\eps}\|^q_T)\big|\hat\mcF_{t}\big]
  \\
  &\leq \frac{C}{K}\check\E^{\check\nu^n}\big[\|\check X^{t,\bx,\eps}\|_T(1+\|\check X^{t,\bx,\eps}\|^q_T)\big|\hat\mcF_{t}\big]
  \\
  &\leq \frac{C}{K}(1+\|\bx\|^{q+1}_{t})
\end{align*}
and similarly for the last term, where $C>0$ does not depend on $K$ or $n$. Concerning the first term, we have
\begin{align*}
\check\E\big[\ett_{E_K}|\hat\kappa^{\check\nu^n}_{T}-\hat\kappa^{\check\nu}_T|\bar\Phi\big|\hat\mcF_{t}\big] &\leq C\check\E\big[|\hat\kappa^{\check\nu^n}_T-\hat\kappa^{\check\nu}_T|(1+K^{q})\big|\hat\mcF_{t}\big],
\end{align*}
where the latter tends to 0, $\hat\Prob$-a.s., as $n\to\infty$ by dominated convergence and the fact that $\hat\kappa^{\check\nu^n}_T\to\hat\kappa^{\check\nu}_T$, $\hat\Prob$-a.s. (see the proof of Lemma~\ref{lem:check-tau-approx} above). For each $K>0$, there is thus a $\check\Prob$-null set $E\subset\hat\mcF$ such that
\begin{align*}
  \lim_{n\to\infty}\check\E\big[|\hat\kappa^{\check\nu^n}_T-\hat\kappa^{\check\nu}_T|\bar\Phi\big|\hat\mcF_{t}\big]&\leq \frac{C}{K}(1+\|\bx\|^{q+1}_{t})
\end{align*}
on $\hat\Omega\setminus E$. Since $K>0$ was arbitrary, we conclude that the left-hand side equals 0, $\check\Prob$-a.s.
\end{proof}

While Section~\ref{subsec:disc-mcR} focused on finding a suitable non-increasing sequence $(\check\tau^{\eps}_n)_{n\in\bbN}$ of $\bbT^{t,\eps}$-valued, approximate optimal stopping times for the truncated randomized game, whose value functions are denoted by $(v^{\mcR,n})_{n\in\bbN}$, we have thus far in Section~\ref{subsec:proof-of-prop} chosen an arbitrary pure control $\hat u^{\varrho,\eps}$ from the set of discretized controls and shown that it can be well approximated by the point process $\check u^{\varrho,\eps}$ associated to a random measure $\check\mu^{\varrho,\eps}$. Since the corresponding density with respect to $\lambda(de)ds$, denoted $\check\nu^{\varrho,\eps}$, is not necessarily bounded we have approximated it by the truncations $\check\nu^{\varrho,\eps,n}:=\check\nu^{\varrho,\eps}\wedge n$. Combining Lemma \ref{lem:check-tau-approx} and Lemma \ref{lem:check-nu-approx} proves the validity of this approximation by showing that $\check J^{\mcR,\eps}(t,\bx;\check\nu^{\varrho,\eps,n},\check\tau^\eps_n)\to\check J^{\mcR,\eps}(t,\bx;\check\nu^{\varrho,\eps},\check\tau^{*,\eps})$, $\hat\Prob$-a.s., as $n\to\infty$. Combined with \eqref{Yn-rel-varrho-3}, we find that
\begin{align}\label{Yn-rel-varrho-4}
 \check\E\big[(v^{\mcR}(t,\bx)- \check J^{\mcR,\eps}(t,\bx;\check\nu^{\varrho,\eps},\check\tau^{*,\eps}))^+\big]\leq \varrho
\end{align}

What remains in connecting the randomised and the original versions of the game is relating the sequence $(\check\tau^{\eps}_n)_{n\in\bbN}$ of stopping times in $\check\mcT^{\mcR,\eps}$, and consequently also $\check\tau^{*,\eps}$, to stopping times in $\hat\mcT^\eps$.

\begin{proof}[Proof of Theorem~\ref{thm:zs-game}]
Since the sequence $(\check\tau^\eps_n)$ consists of $\check\bbF^{\mcR,\eps}$-stopping times, there is by the Doob-Dynkin lemma a sequence of (deterministic) maps $\mathfrak t^{\eps}_{n}:\bfC^d_{0,0}\times\bfU^\eps\to\bbT^{t,\eps}$ such that\footnote{Note that, since $\hat W$ is a $d$-dimensional standard Brownian motion, the paths of $\hat W$ are in $\bbC^d_{0,0}$}
\begin{enumerate}[a)]
  \item The representation $\check\tau^{\eps}_n=\mathfrak t^{\eps}_{n}(\hat W,\Xi^\eps(\check u^{\varrho,\eps}))$ holds $\hat\Prob$-a.s.
  \item For any $s\in [t,T]$, the set $\{(\bold w,\bold u)\in \bfC^d_{0,0}\times \bfU^\eps: \mathfrak t^{\eps}_{n}(\bold w,\bold u)\leq s\}$ is $\mcC^{0,0}_s\otimes\mcF^{\bfU^\eps}_s$-measurable.
\end{enumerate}
Letting
\begin{align*}
  \hat \tau^{\eps}_n:=\mathfrak t^{\eps}_{n}\big(\hat W,(\beta_0\ett_{[0,t+\Delta^{t,\eps^2})}(s)+\hat u^{\varrho,\eps}_{s-\Delta^{t,\eps^2}}\ett_{[t+\Delta^{t,\eps^2},T]}(s):0\leq s\leq T)\big),
\end{align*}
property \emph{b)} ensures that $\hat \tau^{\eps}_n\in\hat\mcT^{\eps}_t$. On the other hand, property \emph{a)} combined with Lemma~\ref{lem:will-conv}.\emph{\ref{enum:disc-equiv})} implies that for each $n\in\bbN$, there is a $\hat\Prob$-null set $E_n$ such that $\hat \tau^{\eps}_n= \check\tau^{\eps}_n$ on $A^{\varrho,\eps}\setminus E_n$, leading to the conclusion there is a $\hat \tau^{\eps}\in\hat\mcT^{\eps}_t$ and a $\hat\Prob$-null set $E^*$ such that $\hat \tau^{\eps}= \check\tau^{*,\eps}$ on $A^{\varrho,\eps}\setminus E^*$. We thus have,
\begin{align*}
  \hat\E\big[(\check J^{\mcR,\eps}(t,\bx;\check \nu^{\varrho,\eps},\check\tau^{*,\eps})-\hat J(t,\bx;\hat u^{\varrho,\eps},\hat \tau^{\eps}))^+\big]& \leq \hat\E\big[(\check J^{\mcR,\eps}(t,\bx;\check \nu^{\varrho,\eps},\check\tau^{*,\eps}) - \check J^{\mcR,\eps}(t,\bx;\check \nu^{\varrho,\eps},\hat\tau^{\eps}))^+\big]
  \\
  &\quad + \hat\E\big[(\hat J(t,\bx;\Xi^{\mcR,\eps}[\check u^{\varrho,\eps}],\hat\tau^{\eps})- \hat J(t,\bx;\hat u^{\varrho,\eps},\hat \tau^{\eps}))^+\big]
  \\
  &\leq \hat\E\big[(\hat J^{\mcR,\eps}(t,\bx;\check \nu^{\varrho,\eps},\check\tau^{*,\eps}) - \check J^{\mcR,\eps}(t,\bx;\check \nu^{\varrho,\eps},\hat\tau^{\eps}))^+\big]+\varrho,
\end{align*}
where the first inequality follows from \eqref{ekv:primal-dual-eps-comp} and the second is due to Lemma~\ref{lem:will-conv}.\emph{a)}.

Concerning the first term, we have that
\begin{align*}
  \check J^{\mcR,\eps}(t,\bx;\check \nu^{\varrho,\eps},\check\tau^{*,\eps}) - \check J^{\mcR,\eps}(t,\bx;\check \nu^{\varrho,\eps},\hat\tau^{\eps})&=\hat\E\Big[\psi(\check\tau^{*,\eps},\check X^{t,\bx,\eps})-\psi(\hat\tau^{\eps},\check X^{t,\bx,\eps})
  +\int_{\hat\tau^{\eps}}^{\check\tau^{*,\eps}} f(r,\check X^{t,\bx,\eps},\check I^{t,\eps}_r)dr\,\Big|\,\hat\mcF_t\Big]
  \\
  &\leq C\hat\E\Big[\ett_{(A^{\varrho,\eps})^c}(1+\|X^{t,\bx,\eps}\|^\rho_T)\,\Big|\,\hat\mcF_t\Big]
  \\
  &\leq C(1+\|\bx\|^\rho_t)\varrho^{1/2}
\end{align*}
and we conclude that
\begin{align}\label{ekv:dual-primal-comp}
  \hat\E\big[(\check J^{\mcR,\eps}(t,\bx;\check \nu^{\varrho,\eps},\check\tau^{*,\eps})-\hat J(t,\bx;\hat u^{\varrho,\eps},\hat \tau^{\eps}))^+\big]& \leq C(1+\|\bx\|^\rho_t)\varrho^{1/2}+\varrho.
\end{align}
Combining \eqref{Yn-rel-varrho-4} with \eqref{ekv:dual-primal-comp} yields
\begin{align*}
 \hat\E\big[(v^{\mcR}(t,\bx)- \hat J(t,\bx;\hat u^{\varrho,\eps},\hat \tau^{\eps}))^+\big]\leq C(1+\|\bx\|^\rho_t)\varrho^{1/2}+2\varrho.
\end{align*}
Moreover, \eqref{ekv:u-varrho-eps-def} holds for all $\tau\in\hat\mcT^\eps$ and since $\hat \tau^{\eps}\in \hat\mcT^\eps$, it follows that
\begin{align*}
 \hat\E\big[(v^{\mcR}(t,\bx)- \hat{\underline v}^{\eps}(t,\bx))^+\big]\leq C(1+\|\bx\|^\rho_t)\varrho^{1/2}+3\varrho.
\end{align*}
As this holds for any $\varrho$ and all sufficiently small $\eps>0$, Lemma~\ref{lem:underv-eps-conv} thus implies that $v^{\mcR}(t,\bx)\leq \underline v(t,\bx)$. Hence, $v^\mcR(t,\bx)=\underline v(t,\bx)=\bar v(t,\bx)$ where, in particular, the latter two can be chosen deterministic.

Finally, we establish continuity. Observe that
\begin{align*}
  \underline v(t,\bx)- \underline v(t',\bx')&=\essinf_{u\in\mcU_t}\esssup_{\tau\in\mcT_t}J(t,\bx;u,\tau)-\essinf_{u\in\mcU_{t'}}\esssup_{\tau\in\mcT_{t'}}J(t',\bx';u,\tau)
  \\
  &\leq\essinf_{u\in\mcU_t}\esssup_{\tau\in\mcT_t}J(t,\bx;u,\tau)-\essinf_{u\in\mcU_{t'}}\esssup_{\tau\in\mcT_{t}}J(t',\bx';u,\tau\vee t')
  \\
  &\leq \esssup_{(u,\tau)\in\mcU\times\mcT_t}\{J(t,\bx;u,\tau)-J(t',\bx';u,\tau\vee t')\}.
\end{align*}
Since the left-hand side is deterministic, we can take the expectation of both sides and then use Lemma~\ref{lem:J-cont} to conclude that $\lim_{(t',\bx')\to(t,\bx)}v(t',\bx')\geq \underline v(t,\bx)$ proving that $\underline v$ and thus also $v^\mcR$ is lower semi-continuous. Moreover, as the limit of a sequence of non-increasing continuous functions $v^\mcR$ is upper semi-continuous and, therefore, continuous.
\end{proof}

%\begin{rem}
%  Note that Theorem~\ref{thm:zs-game} implies that $v^\mcR$ is independent of $\beta_0$.
%\end{rem}

%The dual characterization of the value function $v$ in Theorem~\ref{thm:zs-game}, in terms of an optimal stopping problem for BSDEs with constrained jumps, naturally yields a dynamic programming principle (DPP) through the tower property of the nonlinear expectation $\essinf_{\nu\in\mcV}\E^\nu$. Using continuity properties of $v$, a corresponding primal-side DPP is derived separately in Appendix~\ref{app:sec:DPP}.

\bibliographystyle{plain}
\bibliography{cont-stop-P1_ref}
\end{document}